\documentclass[12pt, reqno, a4paper,oneside]{amsart}

\usepackage{graphicx, epsfig, psfrag}
\usepackage{amsfonts,amsmath,amssymb,amsbsy,amsthm}
\usepackage{bm}
\usepackage{color}
\usepackage{float}
\usepackage{hyperref}
\usepackage{mathrsfs}
\usepackage{longtable}
\usepackage[top=3cm,bottom=3cm,left=3cm,right=3cm]{geometry}
\usepackage{epstopdf}
\usepackage{mathabx}
\usepackage{stmaryrd}
\usepackage{ulem}
\usepackage{scalerel}
\usepackage{enumitem}
\usepackage{algorithm}
\usepackage{subfig}

\usepackage{environments}
\numberwithin{theorem}{section}
\numberwithin{equation}{section}

\begin{document}

\title[Consistency of Coarse-grained Dynamical Atomistic Chain]{A Consistency Study of Coarse-grained Dynamical Chains through A Nonlinear Wave Equation of Mixed Type}

\author{M. Liao}
\address{Mingjie Liao\\
Department of Applied Mathematics and Mechanics\\
University of Science and Technology Beijing\\
No. 30 Xueyuan Road, Haidian District\\
Beijing 100083\\
China}
\email{mliao@xs.ustb.edu.cn}

\author{P. Lin}
\address{Ping Lin\\
Department of Mathematics\\
University of Dundee\\
Dundee, DD1 4HN, Scotland\\
United Kingdom}
\email{plin@maths.dundee.ac.uk}

\thanks{ ML and PL were partially supported by National Natural Science Foundation of China (No. 11771040, 11861131004, 91430106).}

\begin{abstract}
A dynamical atomistic chain to simulate mechanical properties of a one-dimensional material with zero temperature may be modelled by the molecular dynamics (MD) model. Because the number of particles (atoms) is huge for a MD model, in practice one often takes a much smaller number of particles to formulate a coarse-grained approximation. We shall mainly consider the consistency of the coarse-grained model with respect to the grain (mesh) size to provide a justification to the goodness of such an approximation. In order to reduce the characteristic oscillations with very different frequencies in such a model, we either add a viscous term to the coarse-grained MD model or apply a space average to the coarse-grained MD solutions for the consistency study. The coarse-grained solution is also compared with the solution of the (macroscopic) continuum model (a nonlinear wave equation of mixed type) to show how well the coarse-grained model can approximate the macroscopic behavior of the material. We also briefly study the instability of the dynamical atomistic chain and the solution of the Riemann problem of the continuum model which may be related to the defect of the atomistic chain under a large deformation in certain locations. 
\end{abstract}

\maketitle

\section{Introduction}
\label{sec:introduction}

Modelling and simulation of material motion plays an important role in understanding and predicting material defect behaviors, such as dislocation and dynamic fracture \cite{Gaoetal1997,Gaoetal1998b,Gaoetal1998,Tadmor1996,Curtin2003,Rudd2005}. An atomistic level model is necessary since the atomistic bond is a crucial factor in such behaviors. However, a full atomistic system is too large to handle even using the most powerful computers available. This is where coarse-grained methods are particularly useful. Problems in this area are very challenging to analysts 
and many theoretical studies so far are about steady state models. In this paper, we shall consider a dynamical atomistic model and its coarse-grained approximation.

Several approaches have been studied in recent years, most notably the quasicontinuum method \cite{Tadmor1996,Shenoy1998,Miller2002,Shapeev2011}, the virtual internal bond method \cite{Gaoetal1997,Gaoetal1998b,Gaoetal1998,Thiagarajan2004,Linshu2002}, the coarse-grained molecular dynamics \cite{Rudd1998,Rudd2000,Curtarolo2002,Zhou2002,Rudd2005,Zhou2005,Lin2007,Kulkarni2008,XLi2010}, the heterogeneous multiscale method \cite{E2004,E2007c,E2010,E2012}, and the bridging scale method \cite{Tang2006,WKLiu2010}. 
Related approaches have also been proposed for simulations involving stochastic systems \cite{Brandt2003,Katsoulakis2006b,Katsoulakis2006a}. In this paper, we consider only the deterministic, dynamical atomistic (or so-called molecular dynamics) model. However, because the number of atoms under consideration is huge, in practice, a coarse-grained model formulated by a relatively small number of particles is involved in simulation. The question is whether the simulation result is consistent with respect to the grain (or mesh) size and whether such a coarse-grained model could still be able to reproduce macroscopic properties of a material. 
Here, macroscopic properties may be represented through the continuum model associated with the atomistic or coarse-grained model.

It is pointed out in \cite{Gaoetal1997, Gaoetal1998b} and \cite{Gaoetal1998} that the prospects for this type of model in numerical simulations of material properties are highly promising. There are not many theoretical studies available. As an early attempt for theoretical justification, we analyze a one-dimensional model with the nearest neighbor interaction at zero temperature, 
and utilize the Lennard-Jones potential as the atomic interacting potential.
It is believed that considering the interaction of one atom with its nearest neighbors is a good approximation to the original finite range interaction problem. We will see that fundamental mathematical problems associated to this simple case are still far from being solved.

A defect of a large deformation may probably lead to a fracture in a material. One of such fundamental problems is that in the case of a large deformation, the type of the associated continuum model
may change from hyperbolic to elliptic. Owing to this change,  
the model may be called a nonlinear wave equation of mixed type and is ill-posed as indicated in \cite{Schaeffer1990}. 
There is almost no mathematical analysis for such mixed type continuum equations associated with the atomistic model interacted with the Lennard-Jones potential. There are some incomplete results in the sense of measure-valued solutions, but only for the case where the interacting potential is close to quadratic \cite{Demoulini1997,Necas1996}. Fortunately, a defect under a large deformation may be formulated as a specific initial value problem, which is usually called a Riemann problem of the nonlinear wave equation. Therefore techniques in analyzing the Riemann problem can be used in the analysis. 

Our main goal is to study the consistency of the coarse-grained solutions with various grain (mesh) sizes.
However, according to \cite{Balk2001} we may find that the characteristic frequencies of the atomistic MD model and its coarse-grained approximation are variant at different grain sizes. So under a standard error measure, e.g. maximum norm, the coarse-grained solutions cannot have consistency and cannot be an approximation to the solution of the atomistic model.
Also, the coarse-grained solution is not similar to the continuum solution, since at the macroscopic scale (or continuum model) we do not usually see the high frequency oscillations. It is expected that the kinetic energy corresponding to the oscillations is transferred into other energy (cf. \cite{Balk2001}). 
Thus to consider the consistency, a damping/dissipation term needs to be introduced to the atomistic model and its coarse-grained approximation to reduce the oscillations.
One possible dissipation may be a viscous term. Another possible way to lower the oscillations is applying a suitable space average. In \cite{Hou1991} it is demonstrated that for a (hyperbolic) nonlinear conservation law, the solution of a central type spatial discretization has a limit under a space average as the mesh size gets smaller and smaller.
In \cite{Yang2012}, a generalized Irving-Kirkwood formulism for the stress is proposed by systematically incorporating the spatial and temporal averaging into the expression of continuum quantities.
In this paper, besides the viscous dissipation, we also borrow the idea of the space average from \cite{Hou1991} to reduce the oscillations, so that we are able to investigate if the solution of the coarse-grained model is consistent with respect to the grain size and if the coarse-grained solution could approximate macroscopic material properties well.

\noindent\textbf{Outline.} 
In \S~\ref{sec:continuum} and \S~\ref{sec:dynamic} we will introduce the continuum model and the atomistic MD model, and analyze the steady state and the dynamic solution. In \S~\ref{sec:discrete}, the solution of the (dynamical) atomistic model or its coarse-grained approximation and its stability will be studied. We will also analyze and numerically demonstrate the consistency of the  coarse-grained model with different grain sizes in the sense of a space average and the sense of the viscous dissipation in \S~\ref{sec:discrete}. We will conclude our results in \S~\ref{sec:conclusion}. 
\def\R{\mathbb{R}}
\def\d{\textrm{d}} 

\noindent\textbf{Notation.} 
Throughout, $\R$ denotes the real numbers.
 
Let $M$ be the number of particles that are uniformly distributed in $[0,1]$ to form the reference configuration $X$. Then the distances between each pairs of nearest particles are identical and expressed as $\Delta X = \frac{1}{M-1}$. The atomic lattice scale $\varepsilon$ of the MD model is a specific value of $\Delta X$. The position in space is denoted by $x=\phi(X, t)$ (cf. Figure \ref{figs:phi}), where $\phi : [0, 1]\times[0, T]\longrightarrow \R$. In the equilibrium study $x=\phi(X)$ is independent of time $t$ .

If $\phi$ is differentiable of $X$ and $t$, we denote the deformation gradient by $F(X, t)=\frac{\partial\phi}{\partial X}$, ($F(X) = \frac{\d F}{\d X}$ in the equilibrium case), and the velocity by $v(X, t)=\frac{\partial\phi}{\partial t}$. 

Let $\Phi$ denote the site potential, $\Theta(F(X)) = \Phi( F(X), 1)$ stands for the continuum energy density, 
and we denote $\sigma=\Theta^{'}$.

\section{Continuous description and equilibrium solutions}
\label{sec:continuum}

Let $\Phi_{ij}(r_{ij})$ denote the interaction energy of atoms $i$ and $j$, where $r_{ij}=|x_i - x_j|$. We assume that the energy functions $\Phi_{ij}$ between any two atoms are identical and simplify the notation as $\Phi$. As is conventional in application of this aspect, we adopt the so-called Lennard-Jones potential

\begin{equation}
  \Phi(r, r_0) = \frac{A}4 [ - 2 (\frac{r_0}{r})^6 +
  (\frac{r_0}{r})^{12} ], 
  \label{eq:eng1}
\end{equation}
where $r_0$ is the minimal point of the potential function and $A$ is a positive scaling constant.

\begin{figure}[htb]
\begin{center}
	\includegraphics[scale=0.5]{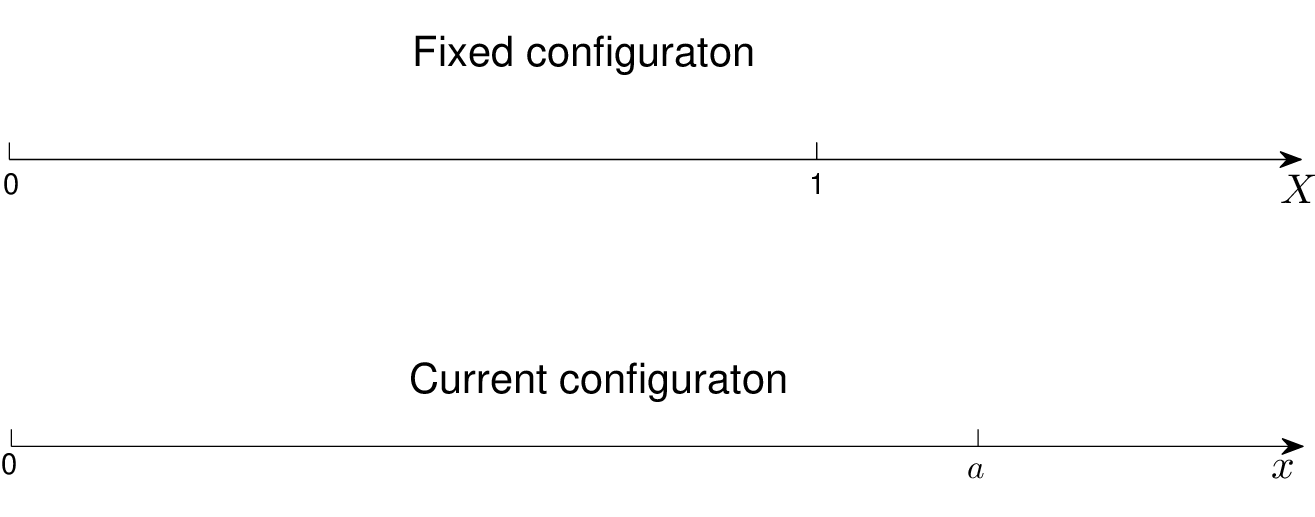}
	\caption{Deformation function $\phi$ maps the fixed configuration to the current configuration with $\phi(0)=0,\phi(1) = a$.}
	\label{figs:phi}
\end{center}
\end{figure}

\def\Ec{E_{\textrm{cont}}}

Lin and Shu \cite{Linshu2002} derived the continuum potential energy  with the virtual internal bond (VIB) method,
\begin{equation}
  \Ec = \int_0^{1}  \Theta\left( F(X)\right) \; dX,
  \label{eq:eng-cont}
\end{equation}
where

\begin{equation}
  \Theta(F(X)) = \Phi( F(X), 1),
  \label{eq:eng-dens}
\end{equation}
represents the energy density of our model. 

Hence the equilibrium configuration can be obtained from solving

\begin{equation}
  \label{eq:min_ec}
  F \in \arg\min \left\{ \Ec(F) \right\},
\end{equation}
subject to boundary conditions:

\begin{equation}
  \phi(0) = 0,\quad \phi(1) = a.
  \label{eq:cont-bc}
\end{equation}

\begin{figure}[htb]
	\centering 
	\subfloat[The graph of $\Theta(F)$.]{ 
		\label{figs:Theta} 
		\includegraphics[width=5cm]{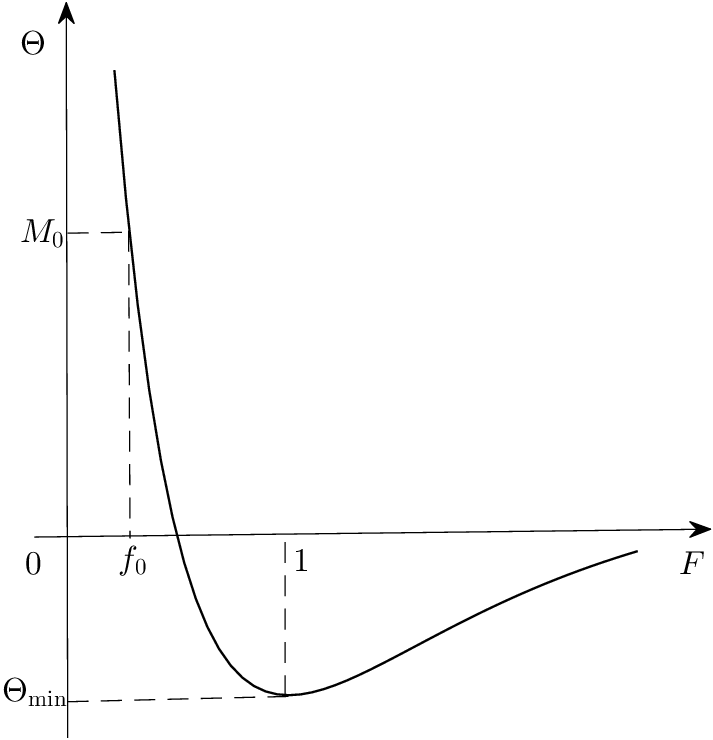}}
		\hspace{0.2cm} 
	\subfloat[The graph of $\sigma(F)$.]{
		\label{figs:sigma}
		\includegraphics[width=7cm]{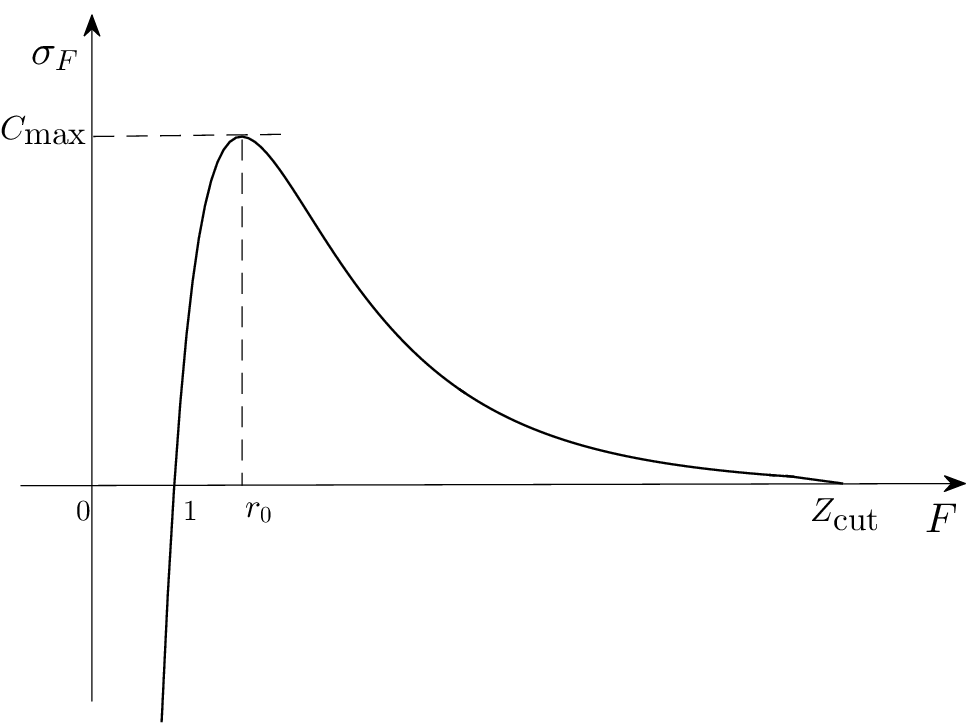}} 
		 \caption{Graphs of $\Theta(F)$ and $\sigma(F)=\Theta'(F)$.}
		\label{figs:ThetaSigma}
\end{figure}

In the case of the Lennard-Jones potential, $\Theta(F)$ and its derivatives are almost zero for a relatively large $F$, say, $F \geq z_{\textrm{cut}}$ (cf. Figure \ref{figs:ThetaSigma}). So we are only interested in the region 

\begin{equation}
  0 \leq F \leq z_{\textrm{cut}}, 
  \label{eq:Fregion}
\end{equation}
and assume that the material is broken down when $F \geq z_{\textrm{cut}}$. Elastic bars and their static theories were studied by James \cite{James1979}. In order to properly describe the total energy, $F(X)=\phi'(X)$ must exist almost everywhere and be measurable. Still, some possibilities of the function space are left open, depending upon the weakness of the derivative. It is indicated in \cite{James1979} that the best alternative seems to be the class of absolutely continuous functions. For each absolutely continuous function $\phi(X)$, there is an integrable function $F(X)$,
such that 
\begin{displaymath}
  \phi(X)=\int^X_0 \; F(Y) \; dY,
\end{displaymath}
and $F(X)$ is existing and integrable almost everywhere. The Euler equation to find the minimizer for the total energy $\Ec(\phi)$ under the class of absolutely continuous functions are given in \cite{James1979}, 

\begin{equation}
  \Theta'(F) = c, \quad \phi(0)=0,\quad \phi(1) = a, 
  \label{eq:cont-stat}
\end{equation}
where $c$ is the integration constant. The minimizer will produce the equilibrium configuration or steady-state solution. 
Note that
\begin{equation}
  \Theta'(F) = 3A [ \frac1{F^7} - \frac1{F^{13}}],
  \label{eq:engp}
\end{equation}
and denote $c_{\textrm{max}} = \max_{F\geq0} \; \Theta'(F)$. By the definition of $z_{\textrm{cut}}$, $\Theta'(z_{\textrm{cut}})$ is almost zero. For simplicity, we let $\Theta'(z_{\textrm{cut}})=0$ and denote
\begin{equation}
  \sigma(F) = \Theta'(F), \quad \mbox{ for } 0 < F \leq z_{\textrm{cut}}. 
  \label{sigmadef}
\end{equation}

The graph of $\sigma$ (Figure \ref{figs:sigma}) shows that there are two solutions for $0 \leq c < c_{\textrm{max}}$ and only one solution for $-\infty < c < 0$.
Let

\begin{displaymath}
  \{F : \sigma(F) = c\} = \left\{
              \begin{array}{ll}
         \{b_{-} \mbox{ or } b_{+}\} \mbox{ (assuming  } b_{-}<b_{+}) ,
  & 0 \leq c < c_{\rm max}; \\
         \{b\} ,  &  -\infty < c < 0. 
              \end{array} \right.
\end{displaymath}

These two cases are considered separately.

\subsubsection*{I. The case $0 \leq c < c_{\textrm{max}}$}

\vspace{0.2cm}

Denote the set measure as $l_{\pm} = \textrm{meas} \{X: F(X) = b_{\pm} \}$, which can be obtained from
the equations

\begin{align}
  l_{-} \;+\; l_{+}  & =  1,  \nonumber \\
  b_{-}l_{-} \;+\; b_{+}l_{+} & =  \phi(1)-\phi(0) = a.  
  \label{eq:lmp1}
\end{align}

It is easy to see from Figure \ref{figs:sigma} that $b_{\pm} \geq 1$. So, from \eqref{eq:lmp1},
\begin{equation}
  b_{+}=b_{+}(l_{-}+l_{+}) \geq a \geq b_{-}(l_{-} + l_{+}) = b_{-} \geq 1. 
  \label{eq:ab1}
\end{equation}

Now for a given $c \in [0, c_{\rm max}]$, $\;b_{\pm}$ are determined. From \eqref{eq:lmp1} we can have
\begin{equation}
  l_{-} = \frac{b_{+} - a}{b_{+} - b_{-}},\quad l_{+}=\frac{a-b_{-}}{b_{+}-b_{-}}.
  \label{eq:lmp-solu1}
\end{equation}

Because $a \in (b_{-},b_{+})$, $\;c$ can be further constrained (cf. Figure \ref{figs:sigma}) to

\begin{equation}
  0 \leq c \leq \sigma(a).
  \label{eq:gamma-restr}
\end{equation}

For any given $a \geq 1$, there are infinitely many solutions satisfying both the boundary conditions \eqref{eq:cont-bc} ($\phi(0)=0$, $\phi(1)=a$), and the following relations

\begin{equation}
  F(X) = \left\{  \begin{array}{cl}
  b_{-}, & \{X\; |\; \textrm{meas}\{X\} = l_{-}\leq 1\} ;\\
  b_{+}, &\{X \;|\; \textrm{meas}\{X\} = l_{+}=1-l_{-}\leq 1\}, 
              \end{array} \right.  
  \label{eq:solution1}
\end{equation}
where $0\leq\sigma(b_{\pm})=c\leq\sigma(a)$.

\subsubsection*{II. The case $-\infty < c < 0$}

\vspace{0.2cm}

In this case, $F = b < 1$. Also $b \cdot 1  =  a$. Thus, 
\begin{equation}
  a = b < 1.  
  \label{eq:al1}
\end{equation}

There is a unique solution:

\begin{equation}
  F(X) = a, \; \mbox{ or } \; \phi(X) = a X, \; X \in [0,1].  
  \label{eq:solution2}
\end{equation}

Next we briefly mention the stability of the equilibrium solution in the sense of energy minimization. In \cite{James1979} a stability concept (which was called metastability) was introduced. It corresponds to the local stability in the
dynamical system context, that is, an equilibrium solution $\psi$ is called metastable or locally stable if

\begin{displaymath}
  E_{\textrm{cont}}(\psi) \leq E_{\textrm{cont}}(\phi),
\end{displaymath} 
where $\phi$ is absolutely continuous and 

\begin{displaymath}
  |\phi'(X) - \psi'(X)| \leq \epsilon_X \quad a.e. 
\end{displaymath}
Note that the constant $\epsilon_X$ should be sufficiently small so that $0\leq \psi'(X)+\epsilon_X \leq z_{\textrm{cut}}$. 

If further we have $E_{\textrm{cont}}(\psi) < E_{\textrm{cont}}(\phi)$ for all $\phi$ belonging to the class of absolutely continuous functions, the equilibrium solution $\psi$ is called absolutely stable. Absolute stability corresponds to the global stability in the dynamical system context. An equilibrium solution $\psi$ is unstable if it is not locally stable. A necessary condition is given in \cite{James1979}, that is (noting $\sigma=\Theta'$),

\begin{displaymath}
  w(F_1,F_2)=\Theta(F_1)-\Theta(F_2)-(F_1-F_2)\sigma(F_2)=\int^{F_1}_{F_2} \sigma(\xi)d\xi - (F_1-F_2)\sigma(F_2) \geq 0 \quad a.e.
\end{displaymath}
$w(F_1,F_2)$ is called Weierstrass excess function. A necessary and sufficient condition for $w(F_1,F_2)\geq0$ is that the area under the graph of $\sigma(\cdot)$ from $F_2$ to $F_1$ exceeds the area of the rectangle of base $(F_1-F_2)$ and height $\sigma(F_2)$. From this criterion, for $\psi'(X) \in [r_0, z_{\textrm{cut}}]$, the solution $\psi$ is unstable, and for $\psi'(X) \in (0, r_0)$, the solution should be locally stable. As for the equilibrium solution $\psi$ obtained in the case of $0 \leq c \leq c_{\textrm{max}}$, if $l_{+} \ne 0$ then $\psi$ is unstable; otherwise, the solution should be locally stable. E and Ming \cite{E2007b} established a sharp criterion for the regime of validity, which is analogous to our results here for one dimensional equations, with some regularity assumptions on the solutions.

According to the explanation of phase transition phenomena in \cite{Balk2001}, two states $b_{\pm}$ 
in the case of $0 \leq c < c_{\textrm{max}}$ do not cause phase transition since one of the states is not locally stable.

\section{Dynamical model}
\label{sec:dynamic}

\def\HCG{\mathcal{H}_{\textrm{CG}}}
\def\VCG{\mathcal{V}_{\textrm{CG}}}
\def\KCG{\mathcal{K}_{\textrm{CG}}}

In the dynamical system context, the total energy of the atomic chain with the nearest neighbor interaction is,

\begin{equation}
  \mathcal{H} = \frac12\sum^{N^a-1}_{i=1}\;m \dot{x}_i^2 + \frac12\sum^{N^a-1}_{i=1}\; \left[  \Phi(|x_{i+1} - x_i|, \varepsilon) + \Phi(|x_{i-1}-x_{i}|,\varepsilon) \right],
  \label{eq:eng-tot-all}
\end{equation}
where $N^a$ is the total number of atoms of the chain, $m$ is the mass of an atom.

The molecular dynamics model is obtained from Hamilton's equations

\begin{equation}
  m \ddot{x}_i =  [ \Phi'(x_{i+1} - x_i, \varepsilon) -
  \Phi'(x_{i} - x_{i-1}, \varepsilon) ], \quad i=1, \cdots, N^{\rm a}-1. 
  \label{eq:md-eqn}
\end{equation}

To simulate dynamical defect movements of materials, we study the dynamical continuum model 
associated with the static model discussed in the previous section

\begin{equation}
  \rho_0 \frac{\partial^2 \phi}{\partial t^2} = \frac{\partial}{\partial X}[\sigma (\frac{\partial \phi}{\partial X})],
  \label{eq:cont-dyn}
\end{equation}
satisfying boundary conditions \eqref{eq:cont-bc} and initial conditions

\begin{equation}
  \phi(X,0) = \phi_0(X),\quad \phi_t(X,0) = v_0(X). 
  \label{eq:cont-ic}
\end{equation}

Here the density $\rho_0=m/r_0$ and $\phi$ is not only a function of $X$ but also a function of $t$. This is the same as the VIB model given in \cite{Gaoetal1997,Gaoetal1998b,Linshu2002}. We should note that $\sigma'=\Theta''$ could be negative when a deformation is large, i.e. $F>r_0$. The equation changes from hyperbolic to elliptic depending on the sign of $\sigma'$. Therefore the dynamical continuum equation is of mixed type and there is almost no mathematical analysis available in general. 

\begin{remark}
\label{rmk:ham}
In some finite temperature dynamical models, for example \cite{Blanc2010,Tadmor2013} and \cite{Kim2014}, the Hamiltonian applying the coarse-grained potential energy has the following form

\begin{equation}
  \label{eq:hcg}
  \HCG = \KCG + \VCG.
\end{equation}

The kinetic energy is (cf. \cite[Equation (53)]{Tadmor2013}),

\begin{equation}
  \label{eq:kcg}
  \KCG = \sum_{i=1}^N\frac{\|M_{i}x_i\|^2}{2M_{i}}, 
\end{equation}
where $M_i$ is the lumped-mass of the uniform coarse-grained mesh.

The coarse-grained potential energy is (cf. \cite[Equation (2.32)]{Blanc2010}),

\begin{equation}
 \label{eq:vha}
  \VCG =  \sum_{e}^{{n}_{\rm elem} } \left[\nu_{e}\Theta(F_{e}) + \frac{K_BT(\nu_{e}-1)}{2}\ln{\Theta''(F_e)} +
  \frac{K_BT(\nu_{e}-1)}{2}\ln{\frac{1}{2\pi K_BT}} + \frac{K_BT}{2}\ln{N} \right]
\end{equation} 
where $n_{\rm elem} = M -1$ is the number of elements, $K_B$ is the Boltzmann's constant, $T$ stands for the temperature, $\nu_e$ is the number of atoms associated with element $e$, $F_e$ means the deformation gradient of element $e$.

Consider the zero-temperature situation, i.e. $T=0$, the Hamiltonian is written as,
\begin{equation}
  \label{eq:sHCG}
  \HCG = \sum_{i=1}^N\frac{\|M_{i}x_i\|^2}{2M_{i}} + \sum_{e}^{{n}_{\rm elem}}\nu_{e}\Theta(F_{e}).
\end{equation}

Thus, the corresponding dynamics model is
\begin{equation}
  \label{eq:cgmd-eqn}
  m\ddot{x}_{i} = \frac{\Theta'(|x_{i+1} - x_i|)-\Theta'(|x_i - x_{i-1}|)}{\Delta X}, \quad i=1, \cdots, N-1 .
\end{equation}

We could see that \eqref{eq:md-eqn} is a special case of \eqref{eq:cgmd-eqn} with $\Delta X = \varepsilon$.

\end{remark}

In simulation, assuming that the defect (a larger deformation) locates uniformly in certain region, the problem could be simplified to the Riemann problem of \eqref{eq:cont-dyn}. As in \cite{Linshu2002}, we apply techniques for the Riemann problem of conservation laws (cf. \cite{Smoller1983,Gaoetal1997}) and for conservation laws of mixed type (cf. \cite{Shearer1993,Slemrod1983} and \cite{Hsiao1990}) to investigate the problem.

We employ concepts and notations from \cite{Linshu2002}. 
Without loss of generality, let $\rho_0=1$, the system can be expressed as

\begin{equation}
  F_t - v_X = 0, \quad v_t - \sigma(F)_X = 0, 
  \label{eq:system}
\end{equation}
or

\begin{equation}
  U_t + P(U)_X = 0,\mbox{   with }\;
  U=\begin{pmatrix} F \\ v \end{pmatrix} \; \mbox{ and } \;
  P(U)=\begin{pmatrix}-v \\ -\sigma(F) \end{pmatrix}, 
\label{eq:system'}
\end{equation}
where $t>0$ and $0\leq X\leq 1$. Consider the system \eqref{eq:system} or \eqref{eq:system'} satisfying the boundary conditions \eqref{eq:cont-bc}. We assume the following Riemann initial conditions

\begin{equation}
  U(X,0) = \left\{ \begin{array}{ll} U_l, & \mbox{if}\; X \in [0, \frac12
  - \delta)\\ U_m, & \mbox{if}\; X \in [\frac12-\delta, \frac12+\delta]\\
  U_r, & \mbox{if}\; X \in (\frac12 + \delta, 1] \end{array} \right.
  \label{eq:riemanninit}
\end{equation}
where $U_l = \left(1, v_l(X, 0) \right)$, $U_m = ( 1+\eta , v_m(X, 0))$ and $U_r = ( 1 , v_r(X, 0))$. Here $\eta$ stands for the initial defect. The initial conditions for $F$ are physically reasonable since a material would usually be in its stable configuration (i.e., $F=1$ where the potential energy reaches the minimum) until the defect is initially forced by external effects. For a solid material, it is also reasonable to assume that the velocities are initially the same at any part of the material, i.e., $v_l(X, 0) = v_m(X, 0) = v_r(X, 0)$. Without loss of generality, we assume

\begin{displaymath}
  v_l(X, 0)=v_m(X, 0)=v_r(X, 0) = 0.
\end{displaymath}

Note that \eqref{eq:system} or \eqref{eq:system'} is a system of conservation laws of mixed type since $\sigma'(F)$ changes its sign when $F$ passing $r_0$. Next we consider the solution of the Riemann problem based on the idea of admissible conditions given in \cite{Shearer1993,Hsiao1990,Smoller1983}. Whenever $\sigma'(F)\geq 0$, we use the notation
\begin{equation}
  \chi(F) = (\sigma'(F))^{1/2} \geq 0. 
  \label{eq:halfpowersigmap}
\end{equation}

The characteristic values of (\ref{eq:system}) or (\ref{eq:system'}) are $\lambda_{\pm} = \pm \chi(F)$, with associated eigenvectors $\nu_{+}(F) = (1, -\chi(F))$ and $\nu_{-}(F)=(1,\chi(F))$, respectively. The system is strictly hyperbolic in the region $D_H= \{F: \; F < r_0 \}$, where the characteristic values are real and distinct, and elliptic in the region $D_E = \{F: \; F \geq r_0 \}$, where the characteristic values are purely imaginary. In our Riemann problem, when $1+\eta < r_0$, $U_l \;, \; U_r \;,\;U_m $ all belong to $ D_H$, when $1+ \eta > r_0$, $U_l \;, \; U_r \in D_H$ and $U_m \in D_E$.

A state $U_l = (F_l, v_l)$ may be joined to a state $\bar{U}=(F, v)$ by a (centered) shock wave with a constant shock speed $s$ if and only if $U_l$, $\bar{U}$ and $s$ are related by the Rankine-Hugoniot conditions (cf. \cite{Smoller1983})

\begin{align}
  & v - v_l = -s (F - F_l),\nonumber\\
  & \sigma(F) - \sigma(F_l) = -s(v -v_l). 
  \label{eq:rhcond}
\end{align}

The function 

\begin{displaymath}
  U(X,t) = \left\{ \begin{array}{ll} U_l ,& \mbox{ if } X \leq st, \\ \bar{U}, & \mbox{ if } X > st \end{array} \right. 
\end{displaymath}
is the shock solution of the system. It is well known that to have a unique solution, a certain type of admissibility criterion\footnote{although the admissible solution may not be unique for conservation laws of mixed type (cf. \cite{Shearer1986}).} needs to be satisfied; for example, the entropy criterion or the viscosity criterion for strictly hyperbolic systems. For problems of mixed type, we adopt the generalized viscosity criterion discussed by Hsiao \cite{Hsiao1990} for the double-well potential (cf. also \cite{Shearer1993}), which generalizes the usual viscosity criterion for strictly hyperbolic systems and constrains the speed of shock solutions wherever shock speeds are defined (cf. \cite{Hsiao1990,Ni1996}). The criterion is, one of the following conditions is satisfied for all $w$ between $F_l$ and $F$,

\begin{equation}
  \frac{\sigma(w)-\sigma(F_l)}{w-F_l} \leq 
  \frac{\sigma(F)-\sigma(F_l)}
  {F - F_l}  \quad \mbox{if }\; s < 0, 
  \label{eq:admis1}
\end{equation}
or

\begin{equation}
  \frac{\sigma(w)-\sigma(F_l)}{w-F_l} \geq 
  \frac{\sigma(F)-\sigma(F_l)}
  {F - F_l}  \quad \mbox{if }\; s > 0, 
  \label{eq:admis2}
\end{equation}
or

\begin{equation}
  s = 0. 
  \label{eq:admis3}
\end{equation}

Admissible shock waves satisfying \eqref{eq:admis1} are called 1-shocks, while those satisfying \eqref{eq:admis2} are called 2-shocks. Shock waves with $s=0$ are called stationary shocks.

Eliminating $v - v_l$ in \eqref{eq:rhcond}, we obtain
\begin{equation}
  s^2 = \frac{\sigma(F)-\sigma(F_l)}{F - F_l} \geq 0. 
  \label{eq:ssquare}
\end{equation}

Shearer \cite{Shearer1993} observed that no states within the elliptic region can be joined by a shock. 

From the admissible criteria \eqref{eq:admis1}-\eqref{eq:admis3} and the Rankine-Hugoniot conditions \eqref{eq:rhcond}, the following sets or curves are defined. For our Riemann problem, $U_l$ belongs to the hyperbolic region $D_H$. The set of states which can be connected to the state $U_l$ by a 1-shock with $U_l$ on the left must lie in the curve

\begin{equation}
  S_1(U_l): \quad v-v_l = -\sqrt{(F-F_l)(\sigma(F) -\sigma(F_l))}, \quad F \leq F_l .
  \label{eq:curveS1}
\end{equation}

The set of states which can be connected to the state $U_l$ by a 2-shock with $U_l$ on the left must lie in the curve
\begin{equation}
  S_2 (U_l): \quad v-v_l = -\sqrt{(F-F_l)(\sigma(F) -\sigma(F_l))}, \quad F_l \leq F \leq \bar{F}(F_l),
  \label{eq:curveS2}
\end{equation}
where $\bar{F}(F_l)$ is defined by $\sigma(\bar{F})=\sigma(F_l)$ if $F_l \in (1,r_0]$ and $\bar{F}(F_l)=z_{\textrm{cut}}$ if $F_l \leq 1$ (The dash curve connecting $R_{1}$ and $z_{\textrm{cut}}$ in Figure \ref{figs:srcurves} describes the values of $\bar{F}(F_{l})$ ). It can be easily verified that 

\begin{equation}
  \frac{d v}{d F} = -(s^2 + \sigma'(F))/2s, 
  \label{eq:scurveslope}
\end{equation}
where $s= \pm \sqrt{(\sigma(F)-\sigma(F_l))/(F - F_l)}$, $s < 0$ on the curve $S_1$ and $s>0$ on the curve $S_2$.

We now consider rarefaction-wave solutions of \eqref{eq:system} or \eqref{eq:system'}, which have the form $U=U(X/t)$. There are two families of rarefaction waves, corresponding to either characteristic family $\lambda_{-}$ or $\lambda_{+}$ (called 1-rarefaction waves or 2-rarefaction waves, respectively). As discussed by Smoller \cite{Smoller1983}, let $\xi=X/t$, with the chain rule, \eqref{eq:system'} could be written as,

\begin{equation}
  (dP-\xi I)U_{\xi} = \mathbf{0}
  \label{eq:charctR}
\end{equation}
where
\begin{displaymath}
  dP=\begin{pmatrix}
    0 & -1 \\
    -\sigma^{\prime}(F) & 0
    \end{pmatrix}	
\end{displaymath}

The solutions of \eqref{eq:charctR} are $\xi=\lambda_{\mp}$, substitute them back into \eqref{eq:charctR} to obtain

\begin{displaymath}
  \begin{pmatrix}
    -\lambda_{\mp} & -1 \\
    -\sigma^{\prime}(F) & -\lambda_{\mp}
  \end{pmatrix}
  \begin{pmatrix}
    F_{\xi} \\
    v_{\xi}
  \end{pmatrix}
  =
  \begin{pmatrix}
    0 \\
    0
  \end{pmatrix}.
\end{displaymath}

While $F_{\xi}\neq 0$, we have $v_{\xi}/F_{\xi}=-\lambda_{\mp}$. Then, if a state $U_l$ is to be joined to a state $U$ by a (centered) rarefaction wave, with the state $U$ on the right of the wave, it is necessary and sufficient that $U$ lies on the integral curve 

\begin{equation}
 \frac{dv}{dF} = -\lambda_{\mp}(F) = \pm \chi(F),
 \label{eq:rcurveslope}
\end{equation}
and $\lambda_{\mp}(F)$ is increasing along this curve from $F_l$ to $F$. By integrating (\ref{eq:rcurveslope}) through a fixed state $U_l$, the set of states $U$, to which $U_l$ may be joined by a rarefaction wave, forms curves $R_1$ (1-rarefaction waves) and $R_2$ (2-rarefaction waves), i.e.,

\begin{equation}
  R_1 (U_l): \quad v - v_l = \int^{F}_{F_l} \chi(y)dy, \quad F_l \leq F \leq r_0, 
  \label{eq:curveR1}
\end{equation}
and

\begin{equation}
  R_2 (U_l): \quad v-v_l = -\int^{F}_{F_l} \chi(y)dy, \quad F \leq F_l .
  \label{eq:curveR2}
\end{equation}

From \eqref{eq:system} and \eqref{eq:charctR}, we could also obtain

\begin{equation}
  F\left(\frac{X}{t}\right) = (\sigma')^{-1}\left(\left(\frac{X}{t}\right)^2\right) .
  \label{eq:soluF}
\end{equation}

We can easily verify $dv/dF > 0$ and $d^2v/dF^2 <0$ for the curve $R_1$, and $dv/dF <0$ and $d^2v/dF^2 >0$
for the curve $R_2$. Figure \ref{figs:srcurves} shows these curves in the $v$-$F$ plane.

\begin{figure}[htb]
\begin{center}
	\includegraphics[scale=0.6]{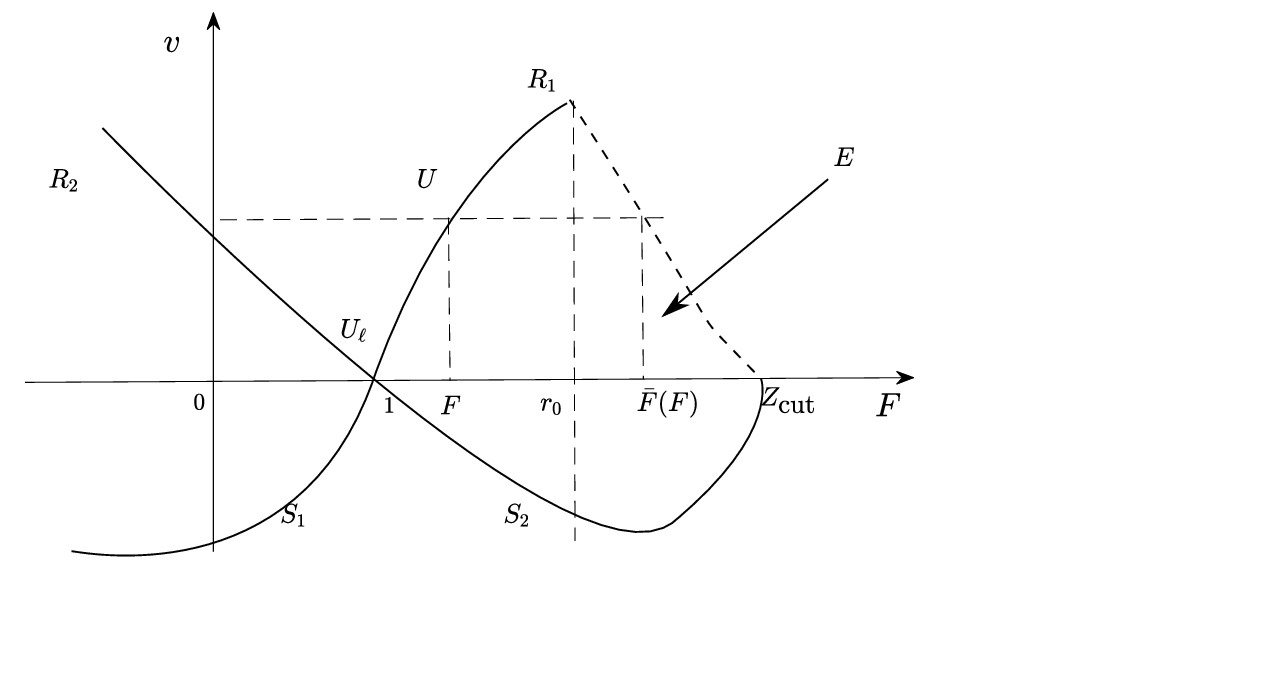}
	\caption{Illustration of curves $S_1(U_l)$, $S_2(U_l)$, $R_1(U_l)$ and $R_2(U_l)$.}
	\label{figs:srcurves}
\end{center}
\end{figure}

\begin{remark}
\label{rmk:3.2}
Since $\sigma(F)\to\sigma(F_l)$ as $F\to z_{\textrm{cut}}$, according to \eqref{eq:ssquare} $s\to 0$ as $F\to z_{\textrm{cut}}$. By \eqref{eq:scurveslope} and some calculation, we obtain

\begin{displaymath}
  \frac{d v}{d F} = -\frac{s}{2}-\frac{\sigma^{\prime}(F)}{2s} \to +\infty \quad \textrm{as} \quad s \to 0. 
\end{displaymath}
The positive sign of infinity comes from $\sigma^{\prime}(F)<0$ when $F\to z_{\textrm{cut}}$.
\end{remark}

Using notations from \cite{Smoller1983}, we define

\begin{displaymath}
  W_i (\bar{U}) = S_i(\bar{U}) \cup R_i(\bar{U}), \quad i=1,2.
\end{displaymath}

For a fixed $U_l \in \R^2$, we consider the family of curves 

\begin{displaymath}
{\mathcal F} = \{W_2(\bar{U}): \bar{U} \in W_1(U_l)\}.
\end{displaymath}

The curves of $\mathcal{F}$ has the following property.

\begin{lemma}
\label{lem:3.3}
For $U_l=(1,0)$, the curves in ${\mathcal F}$ smoothly fill a part of $v$-$F$ plane without gaps or self-intersections.
\end{lemma}

\begin{proof}
Let $\bar{U}$ lie in $W_1(U_l)$. Each curve is the solution of a smooth ordinary differential equation with initial values $\bar{U}$ varying smoothly. By the continuous dependence of the ODE solution on the initial data and the transversality of the curve $W_1(U_l)$ with curves in ${\mathcal F}$, there are no gaps. The proof of no-self-intersection depends on the position of connecting state $U$. If $U$ lies in hyperbolic region then it has been proved (cf. \cite{Smoller1983}). If $U$ lies in the elliptic region $E$ (see Figure \ref{figs:srcurves}), we proceed as follows. Suppose that there are two curves both passing through $U \in E$ starting from $U_1=(F_1,v_1) \in W_1(U_l)$ and $U_2=(F_2, v_2) \in W_1(U_l)$, respectively. Without loss of generality let $v_1 < v_2$ and $F_1 < F_2$ since the curve $W_1(U_l)$ is strictly increasing. Then the vertical difference of these two curves

\begin{align*}
d(U_1,U_2)&= v_1-v_2 - (s_{+}(F_1,F)(F-F_1) 
- s_{+}(F_2,F)(F-F_2))\\
& = v_1-v_2 -(s_{+}(F_1,F)-s_{+}(F_2,F))
(F-F_1) + s_{+}(F_2,F)(F_1-F_2),
\end{align*}
where $s_{+}(p_1,p_2)=\sqrt{\frac{\sigma(p_1)-\sigma(p_2)}{p_1-p_2}}$ is the square root of the slope of the chord connecting points $p_1$ and $p_2$ on the curve of $\sigma$ and $F \geq r_0$ since $U \in E$. From Figure \ref{figs:sigma}, it is easy to see that $s_{+}(F_1,F) > s_{+}(F_2,F)$ for $F \geq r_0$. Hence all terms in the expression of $d(U_1,U_2)$ are negative and no self-intersection is possible. 

\end{proof}

Lemma \ref{lem:3.3} implies there exists one unique curve  $W_2(\bar{U})$ of ${\mathcal F}$ passes through each state $U=(F,0), 1\leq F\leq z_{\rm cut}$. Then, the solution of the Riemann problem is given by first connecting $\bar{U}_1$ to $U_l$ on the right by a 1-(shock, or rarefaction) wave, then connecting $U$ to $\bar{U}_1$ on the right by a 2-(shock, or rarefaction) wave.
The particular type of occurring waves depends on the position of $U$. In our case, $U=U_m$ is on the $F$ axis. The solution of the Riemann problem is illustrated in Figures \ref{figs:rsolua} with $\eta > 0$ (cf. \cite{Smoller1983}). 

\begin{figure}[htb]
\begin{center}
	\includegraphics[scale=0.6]{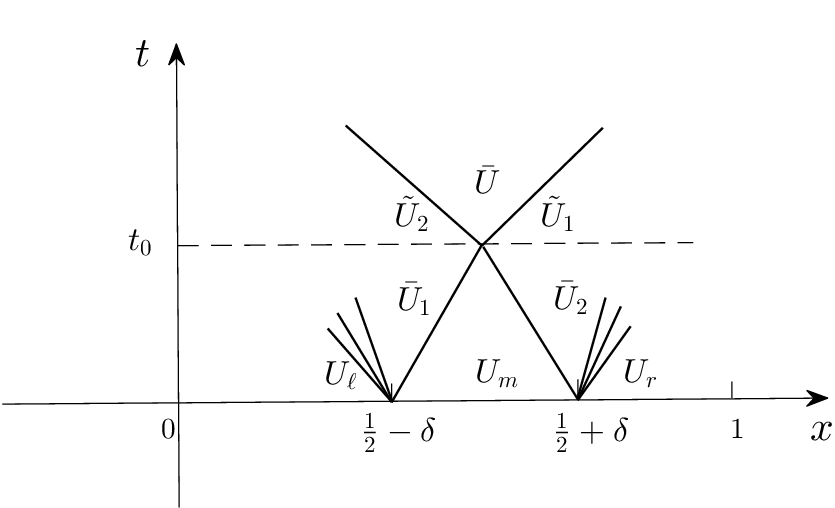}
	\caption{Solution of the Riemann problem for $\eta>0$.}
	\label{figs:rsolua}
\end{center}
\end{figure}

The Riemann solution with the initial data $U_m$ in $1/2<X < 1/2 +\delta$ and $U_r$ in $X>1/2 + \delta$ can be obtained by symmetry and is also depicted in Figures \ref{figs:rsolua}. There is an interaction of shocks at $X=1/2$ after time $t_0$. The solution of the time interval $0\leq t < t_0$ had been considered in \cite{Linshu2002}. In this paper, by considering an additional Riemann problem with initial data $\bar{U}_1$ and $\bar{U}_2$ at time $t_0$, we construct the solution of $t_0\leq t <T$. Time $T$ represents the moment when the wave reaches the outermost boundaries.

\section{Atomistic (MD) model and its coarse-grained approximation - consistency studies}
\label{sec:discrete}

We have analyzed the continuum model \eqref{eq:cont-dyn} in the previous section. In this section, we consider a typical spatially discretized model of \eqref{eq:cont-dyn} under a uniform mesh:
\begin{equation}
  \rho_0 \frac{d^2 \phi_j}{dt^2}  =
  \frac1{\Delta X}[\sigma(\frac{\phi_{j+1}-\phi_j}{\Delta X}) - \sigma(\frac{\phi_j -
  \phi_{j-1}}{\Delta X})], \quad j=1, 2, \cdots, M-1, 
  \label{eq:dis-pde}
\end{equation}
and
\begin{equation}
  \phi_0 = 0, \quad \phi_{M}= a, 
  \label{eq:dis-pde-bc}
\end{equation}
where $\Delta X = \frac{1}{M-1}$. Such a discrete model with $\Delta X =  \varepsilon$ corresponds to a mass-spring chain (cf. \cite{Balk2001}) or the MD model \eqref{eq:md-eqn}. 
Compared with the function $\sigma=\Theta'$ in \eqref{eq:engp}, the discrete model \eqref{eq:dis-pde} is actually the same as the nearest neighbor interaction MD model \eqref{eq:md-eqn}, since if $\Delta X = \varepsilon$, $\frac{1}{\varepsilon}\sigma\left(\frac{r}{\varepsilon}\right)=\Phi'(r, \varepsilon)$.
So we can regard the nearest neighbor interaction MD model as a discretization of the continuum model with a very fine mesh size $\varepsilon$.

In the rest of this section, we will briefly study the instability of the model (the atomistic model or its coarse-grained approximation) with a large deformation and consider the consistency of the coarse-grained model (discrete model with a large mesh size).

\subsection{Stability effects with a large deformation}
\label{sec:stability}

We have mentioned that the continuum model \eqref{eq:cont-dyn} is of mixed type, since the sign of $\sigma'$ would switch to negative at a large deformation gradient $F>r_0$. When $\frac{\phi_j - \phi_{j-1}}{\Delta X}$ is large, this may also happen to the discretization \eqref{eq:dis-pde}. With this sign change, the spatial operator or the matrix of the discrete problem would have positive eigenvalues. This fact implies that the linearized problem would be unstable, i.e. an exponentially increasing solution with respect to time $t$. Any numerical method would fail in this situation. Fortunately, 
in this section we will show by an example that this exponential increase is only related to the linearization of the problem. For the original nonlinear problem, the situation is better --- its solution is always bounded.

We linearize the discrete model \eqref{eq:dis-pde} at a time dependent state $F_j=F(X_j,t)=\frac{\partial\phi}{\partial X}(X_j,t)$, 

\begin{align*}
\sigma(\frac{\phi_{j+1}-\phi_{j}}{\Delta X}) & \approx \sigma(F_j) + \sigma'(F_j)(\frac{\phi_{j+1}-\phi_{j}}{\Delta X}-F_j), \\
\sigma(\frac{\phi_{j}-\phi_{j-1}}{\Delta X}) & \approx \sigma(F_j) + \sigma'(F_j)(\frac{\phi_{j}-\phi_{j-1}}{\Delta X}-F_j).
\end{align*}

Substitute them into \eqref{eq:dis-pde} to have
\begin{equation}
  \rho_0 \frac{d^2 \phi_j}{dt^2}  - \sigma'(F_j) \frac{\phi_{j+1}-2\phi_j+\phi_{j-1}}{(\Delta X)^2}=0, \quad j=1, 2, \cdots, M-1. 
  \label{eq:dis-linear}
\end{equation}

Notice that the right hand side of \eqref{eq:dis-linear} should be the truncation error $O((\Delta X)^2)$. We freeze the variable coefficient $\sigma'(F_j)$ to study the stability. Let

\begin{equation*}
\phi_j(t) = \bar{\phi}e^{\lambda t + ikj\Delta X},
\end{equation*}
with some constant $\bar{\phi}$. $\lambda$ satisfies,
\begin{equation}
\lambda_{\pm} = \sqrt{\frac{-4\sigma'}{(\Delta X)^2}\sin^2(\frac{k\Delta X}{2}) }.
\label{eq:dis-linear-l}
\end{equation}

So if $\sigma'<0$, the solution increases exponentially to infinity as $t$ increases. Fortunately, in practice, if a fracture is caused by a defect, the fracture time may be short, and during the process $|\sigma'|$ becomes smaller and smaller. This would make the exponential
growth of a solution less serious. In addition, for the original nonlinear semi-discrete model, we would not see any exponential growth of the solution. 

The reason is that the original nonlinear semi-discrete model is a Hamiltonian system and we will show below that the energy is bounded for the perturbed problem, so its solution is bounded. Hence the solution is at least not exponentially growing.

Consider \eqref{eq:dis-pde} with boundary conditions \eqref{eq:dis-pde-bc} and the perturbed initial conditions about the equilibrium solution $a X_i$:

\begin{equation}
  \phi_i(0) = a X_i + O(\varepsilon),\quad (\phi_i)_t(0) =  O(\varepsilon).
  \label{eq:dis-ic'}
\end{equation}

We consider the $O(\varepsilon)$ perturbation only to compare with the linearized case. 
For perturbations which are not small, the following argument holds as well. 

Define

\begin{equation} 
E^d(\phi^h,\phi^h_t) = \frac{\rho_0}{2M} \sum_{j=1}^{M-1}\; (\phi_j)_t^2 + \frac1{M}\sum_{j=1}^{M}\; \Theta(\frac{\phi_j - \phi_{j-1}}{\Delta X}),
\label{eq:dis-Ed}
\end{equation}
where $\phi^h = (\phi_1, \cdots, \phi_{M-1})^T$ and $\phi^h_t = ((\phi_1)_t, \cdots, (\phi_{M-1})_t)^T$. The following theorem states that the energy $E^d$ is bounded for all time. The boundedness of the energy ensures that the material would not break down. 

\begin{theorem}
\label{thm:dis-Ed}
For the solution of \eqref{eq:dis-pde} with boundary conditions \eqref{eq:dis-pde-bc} and initial conditions \eqref{eq:dis-ic'}, the energy $E^d$ does not increase with respect to the time $t$ or
\begin{equation*}
E^d(\phi^h, \phi^h_t) = E^d (\phi^h(0), \phi^h_t(0)).
\end{equation*}
\end{theorem}

\begin{proof}
Multiplying (\ref{eq:dis-pde}) by $(\phi_j)_t$ and summing the equations from $1$ to $M-1$, we have (noting $\sigma=\Theta'$)

\begin{align*}
& \rho_0 \frac12 (\sum_{j=1}^{M-1}\; (\phi_j)^2_t)_t\\
&  = \frac1{\Delta X} \sum_{j=1}^{M-1} \;\sigma(\frac{\phi_{j+1}-\phi_j}{\Delta X})(\phi_j)_t - \frac1{\Delta X} \sum_{j=1}^{M-1} \; \sigma(\frac{\phi_{j}-\phi_{j-1}}{\Delta X})(\phi_j)_t \\
&  = \frac1{\Delta X} \sum_{j=2}^{M} \; \sigma(\frac{\phi_{j}-\phi_{j-1}}{\Delta X})(\phi_{j-1})_t - \frac1{\Delta X}
\sum_{j=1}^{M-1} \; \sigma(\frac{\phi_{j}-\phi_{j-1}}{\Delta X})(\phi_j)_t \\
&  = - \sum_{j=1}^{M} \; \sigma(\frac{\phi_{j}-\phi_{j-1}}{\Delta X})(\frac{\phi_j - \phi_{j-1}}{\Delta X})_t = - \sum_{j=1}^{M} \; \Theta(\frac{\phi_{j}-\phi_{j-1}}{\Delta X})_t .
\end{align*}

By the definition of $E^d$ \eqref{eq:dis-Ed}, we obtain
\begin{equation}
  (E^d)_t = 0. 
  \label{eq:eng-ineq'}
\end{equation}

Hence, the energy $E^d$ does not increase with respect to the time $t$ or
\begin{equation}
  E^d(\phi^h, \phi^h_t) = E^d (\phi^h(0), \phi^h_t(0)). 
  \label{eq:eng-bound}
\end{equation}
Here $\phi^h = (\phi_1, \cdots, \phi_{M-1})^T$ and $\phi^h_t = ((\phi_1)_t, \cdots, (\phi_{M-1})_t)^T$. Hence the energy $E^d$ is bounded for all $t$. 
\end{proof}

The following theorem gives the bound and also the global existence of the solution of \eqref{eq:dis-pde} based on the above energy conservation Theorem \ref{thm:dis-Ed}. The existence of a global solution is the prerequisite of the consistency study.

\begin{theorem}
\label{thm:4.3}
For the solution of \eqref{eq:dis-pde} with boundary conditions \eqref{eq:dis-pde-bc} and initial conditions \eqref{eq:cont-ic}, we have
\begin{equation*}
  \|v^h\|_2 \leq \frac2{\rho_0}[E^d(\phi^h(0),v^h(0)) + |\Theta_{\rm min}|]=O(1), \; \|\phi^h\|_{\infty} \leq a , 
\end{equation*}
where $\phi^h=(\phi_1,\cdots,\phi_{M-1})^T$, $v^h=\phi^h_t= ((\phi_1)_t, \cdots, (\phi_{M-1})_t)^T$, $\|v^h\|_2 = (\frac1{M} \sum_{j=1}^{M-1}\; (\phi_j)_t^2)^{\frac12}$ and $\Theta_{\rm min}$ represents the minimum value of $\Theta$ (cf. Figure \ref{figs:Theta}). Hence there exists a global solution $\phi^h(t) \in \mathcal{C}^1$ of \eqref{eq:dis-pde}, \eqref{eq:dis-pde-bc}, \eqref{eq:cont-ic}.
\end{theorem}

\begin{proof}
The bound for $\|v^h\|_2$ results immediately from the definition of $E^d$, the energy estimate (\ref{eq:eng-bound}) and $-\Theta \leq |\Theta_{\rm min}|$. For the boundedness of $\phi_j$, we only need to show that the
atoms do not cross one another since if so, $\phi_j$ will stay between two end atoms; i.e., between $\phi_0=0$ and $\phi_M=a$. Indeed, if there is a crossing, say $\phi_{i_0}=\phi_{i_0-1}$, then from the expression of the energy $E^d(\phi^h,v^h)$, we easily see that it goes to $\infty$ at the crossing point. This is a contradiction with the boundedness (\ref{eq:eng-bound}). So $0\leq \phi_j \leq a$. This proves the boundedness of $\|\phi^h\|_{\infty}$. We then conclude that there exists a global solution $\phi_j(t) \in {\bf C}^1$ for all $t>0$\footnote{The justification of ${\bf C}^1$ is that $\Theta \in {\bf C}^1$ if $\phi_j$, $j=0,1,\cdots,M$, do not cross each other.} since the solution must go to infinity if it cannot expand at some time $t$ (cf. \cite{stuart} Theorem 2.1.4]). 
\end{proof}

So unlike the linearized problem, the solution of the original nonlinear problem will not grow to infinity. 

\subsection{Consistency of the coarse-grained model with different grain sizes}

\def\E{\mathcal{E}}

In this section, we consider whether the coarse-grained model of \eqref{eq:dis-pde} could be a good approximation to the MD model. Since the MD model is the same as the coarse-grained model with mesh size $\varepsilon$, if the coarse-grained model is a good approximation to itself with a smaller mesh size, then perhaps the coarse-grained  model  would be a good approximation to  the MD model. Unfortunately, this is not usually the case. As pointed out in \cite{Balk2001}, if we treat the discrete model as a chain of mass-spring elements in a fixed length $[0,a]$ and assume the spring has constant static stiffness, then the stiffness $k$ of each spring is proportional to the number of mass-spring elements, say $M$, in the chain, and the mass $m$ of each element is inversely proportional to $M$. Therefore the characteristic frequency $\omega$
is proportional to $M$:
\begin{displaymath}
  \omega = \sqrt{k(M)/m(M)} \sim M.
\end{displaymath}

Thus the coarse-grained model is generally not a good approximation to itself with a smaller mesh size, since the characteristic frequencies of them may be largely different (cf. \cite{Cohen2005}). From the viewpoint of shock capturing, the coarse-grained model is not a good approximation to the dynamical model either. The discrete model involves a central difference scheme, which is not recommended in finding the shock solution since it would cause oscillations near the shock discontinuity that would ultimately destroy the solution. To show this, consider the model with $\rho_0=1$, $3A=0.25$ ($A$ is the constant in \eqref{eq:eng1} and \eqref{eq:engp}) and solve the semi-discrete differential equation by MATLAB built-in ODE solver ODE45. We choose the following initial values

\begin{equation}
  F = \left\{ \begin{array}{ll} s_1 & \mbox{if}\; X\in [0, \frac12
  - \delta],\\ s_2 & \mbox{if}\; X\in(\frac12 - \delta, \frac12 + \delta),\\
  s_1 & \mbox{if}\; X\in [\frac12 + \delta, 1]. \end{array} \right.
  \label{eq:initialvalueform}
\end{equation}

Taking $s_1=0.6$ and $s_2=1.2$, Figure \ref{figs:osc} indicates that the frequencies of oscillations of $F$ and $v$ increase with rising $M$, for $M\in\{16,\;32,\;64,\;128\}$. 

\begin{figure}[htb]
\begin{center}
	\includegraphics[scale=0.6]{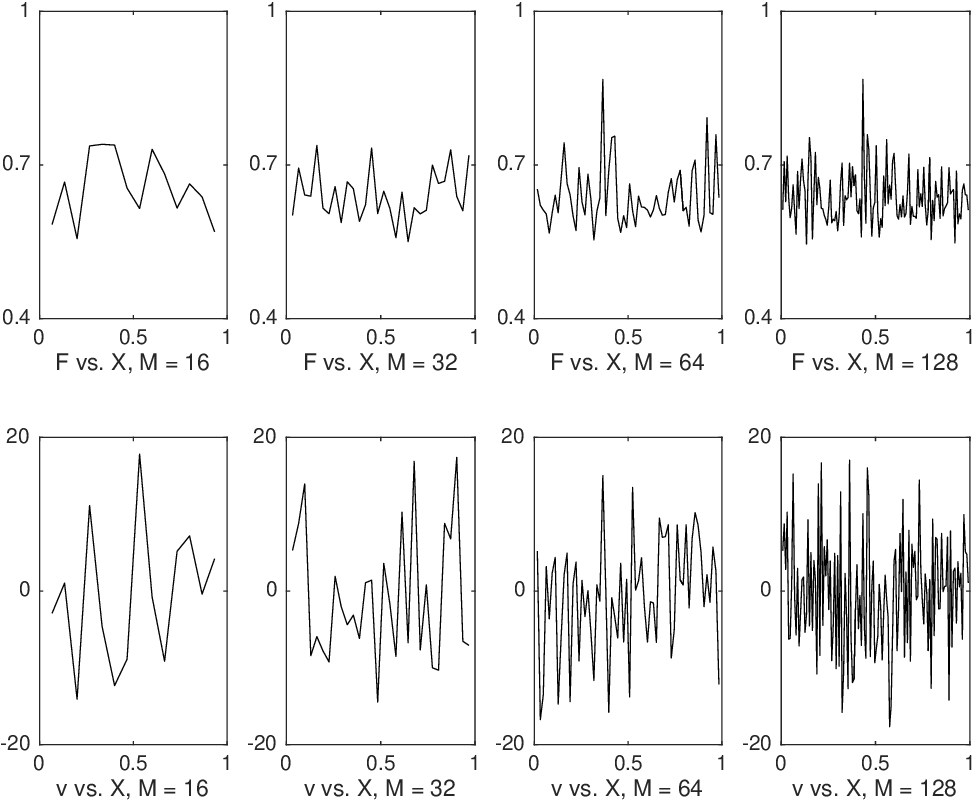}
	\caption{Graphs of $F$ vs. $X$ in the top line and $v$ vs. $X$ in the bottom line at $t=1$ for $M\in\{16,\;32,\;64,\;128\}$ with $s_1=0.6, \ s_2=1.2$ and $\delta=0.01$.}
	\label{figs:osc}
\end{center}
\end{figure}

On the other hand, as $M \rightarrow \infty$ (continuum limit) the characteristic frequency goes to infinity. The high-frequency oscillation becomes ``invisible'' in the limit. This suggests that to discuss the coarse-grained approximation, the oscillations should be reduced in some measures. In this paper, we consider two types of such measures.

\subsubsection{Coarse-grained solution in the measure of the viscous dissipation}
\label{sec:viscousDamping}

Balk et. al. \cite{Balk2001} interpreted that the kinetic energy corresponding to high-frequency oscillations will be transferred into heat. So one type of measure is to add an energy-dissipative term. According to \cite{Hsiao1990}, the admissible criterion of shock we adopted is derived from the relatively simple viscous system

\begin{align}
  &  \rho_0 v_t =  \sigma(F)_X + \mu v_{XX}, \label{eq:viscous1}\\
  &  F_t - v_X = 0,     \label{eq:viscous2}
\end{align}
or

\begin{equation}
  \rho_0 \frac{\partial^2 \phi}{\partial t^2} = \frac{\partial}{\partial X}[ \sigma(\frac{\partial \phi}{\partial X})] + \mu   \frac{\partial^3 \phi}{\partial t \partial X^2}.
  \label{eq:viscous}
\end{equation}

This kind of viscous term was also used by M\'{a}lek et. al. \cite{Necas1996} to discuss other non-convex potential energy. We thus add the spatial discretization of $\mu \frac{\partial^3 \phi}{\partial t \partial X^2}$ as a dissipative term to \eqref{eq:dis-pde}, and obtain,

\begin{equation}
  \rho_0 \frac{\partial^2 \phi_j}{\partial t^2} = \frac{1}{\Delta X}[\sigma(\frac{\phi_{j+1}-\phi_j}{\Delta X}) -\sigma(\frac{\phi_j - \phi_{j-1}}{\Delta X})] + \mu \frac{d}{dt} (\frac{\phi_{j+1}-2\phi_j + \phi_{j-1}}{(\Delta X)^2}).
  \label{eq:dis-viscous}
\end{equation}

The stability effect of a large deformation to \eqref{eq:dis-viscous} was studied in \cite{Linshu2002}. The exponential growth also happens for the linearization of \eqref{eq:dis-viscous}. Similarly to \S~\ref{sec:stability}, 
the energy $E^d$ of \eqref{eq:dis-viscous} does not increase with respect to time, which ensures the boundedness and global existence of the solution with the viscous dissipation. This implies that the exponential growth phenomenon for the linearized model will not occur for the above equation.

Analogous to the error analysis in \cite{Linshu2002}, we define
\begin{equation}
H^{d} = \frac{1}{2M} \sum_{j=1}^{M-1}\; (\E_j)_t^2 + \frac{1}{2M} \sum_{j=1}^{M}\left(\frac{\E_j-\E_{j-1}}{\Delta X}\right)^2,
\label{eq:hd}
\end{equation}
where $\E_j(t) = \phi(X_j,t)-\phi_R(X_j,t)$, as a measure of the distance between solutions, and $\phi_R(X_j,t)$ represents the reference solution.

\begin{figure}[htb]
\begin{center}
	\includegraphics[scale=0.6]{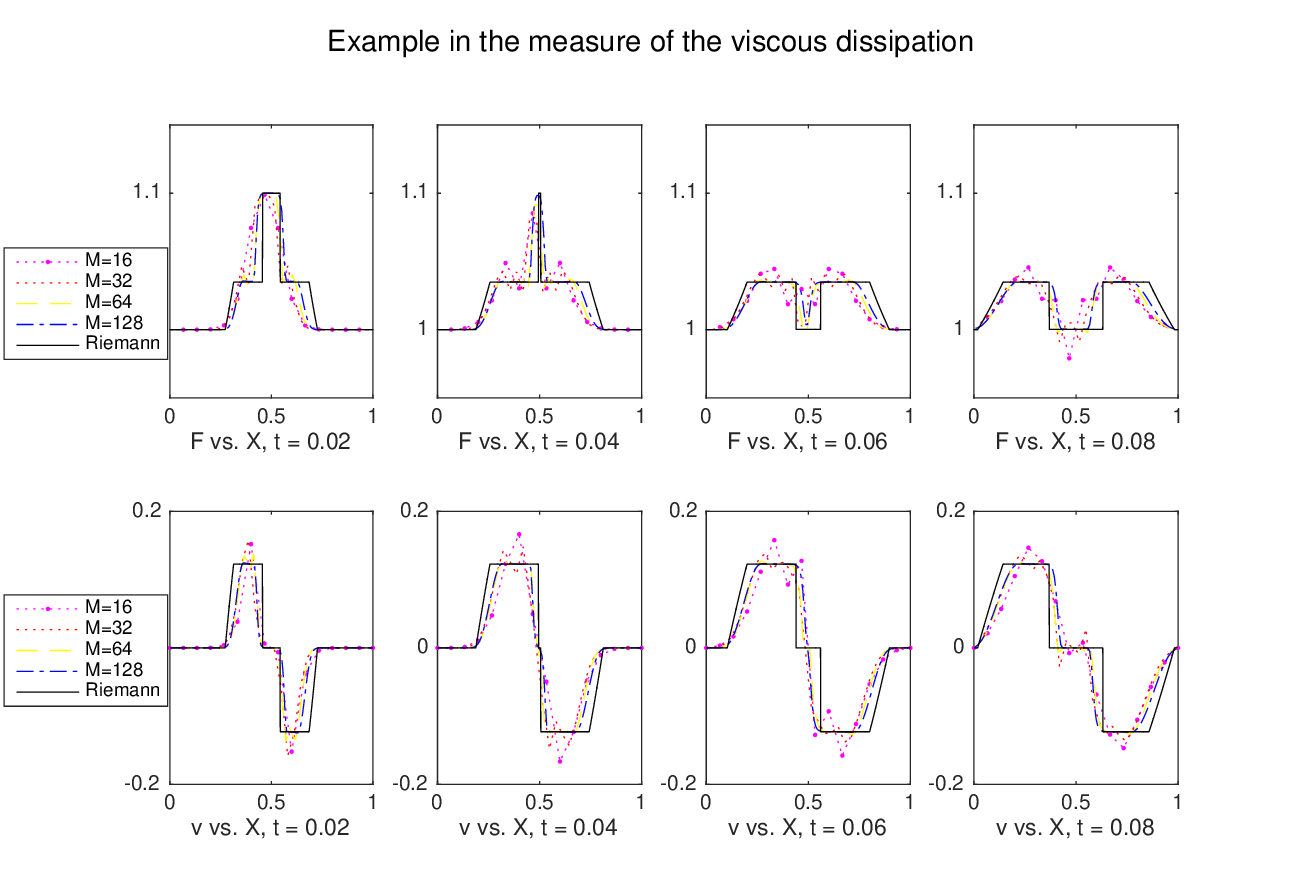}
	\caption{Numerical results with the viscous dissipation: $F$  vs. $X$ (top line) and $v$ vs. $X$ (bottom line) for $t=i \Delta t$, \ $i=1, \cdots, 4$, and $\Delta t=0.02$ with $s_1=1$, $s_2=1.1<r_0=1.1087$, $\delta = 0.1$, and $\mu = 0.01$.}
	\label{figs:vsd1-1}
\end{center}
\end{figure}

\begin{figure}[htb]
\begin{center}
	\includegraphics[scale=0.6]{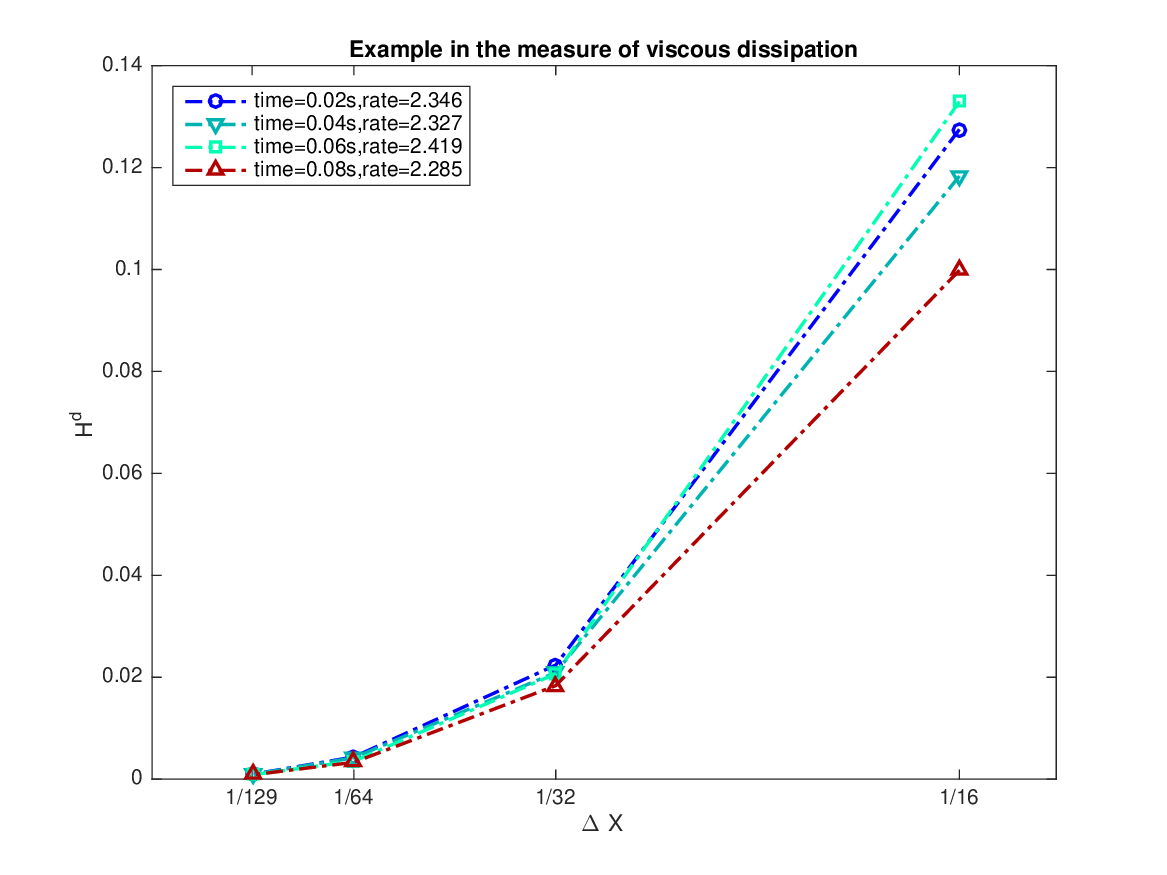}
	\caption{$H^d$  vs. $\Delta X$ in the viscous dissipation sense for $\Delta X = \frac{1}{M}, \ M \in \{16, 32, 64, 128\}$ at $t=i\Delta t, \ i = 1,\cdots, 4$, and $\Delta t = 0.02$ with $s_1=1$, $s_2=1.1<r_0=1.1087$, $\delta = 0.1$, and $\mu=0.01$.}
	\label{figs:convVsd1-1}
\end{center}
\end{figure}

\begin{figure}[htb]
\begin{center}
	\includegraphics[scale=0.6]{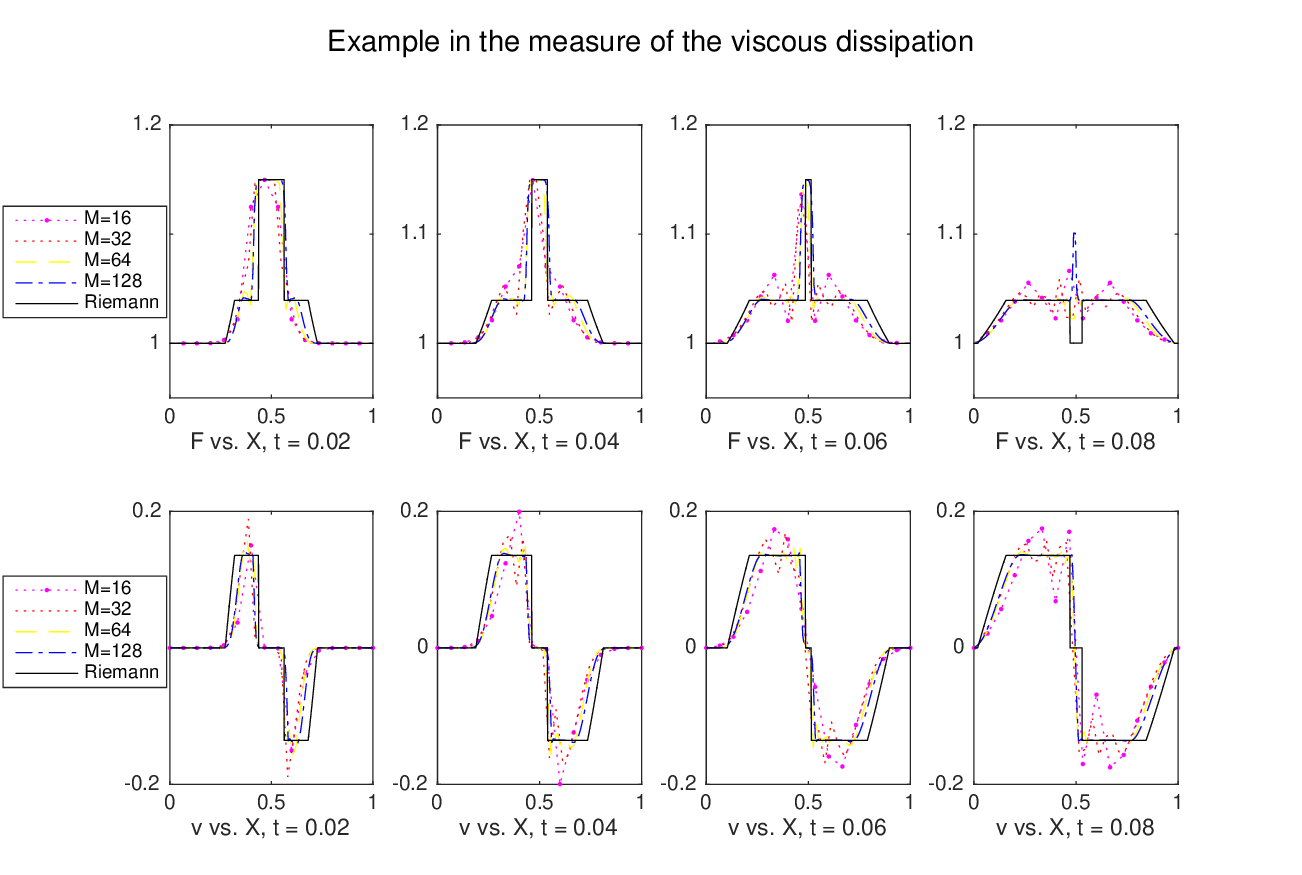}
	\caption{Numerical results with the viscous dissipation: $F$  vs. $X$ (top line) and $v$ vs. $X$ (bottom line) for $t=i \Delta t$, \ $i=1, \cdots, 4$, and $\Delta t=0.02$ with $s_1=1$, $s_2=1.15>r_0=1.1087$, $\delta = 0.1$, and $\mu=0.01$.}
	\label{figs:vsd1-15}
\end{center}
\end{figure}

\begin{figure}[htb]
\begin{center}
	\includegraphics[scale=0.6]{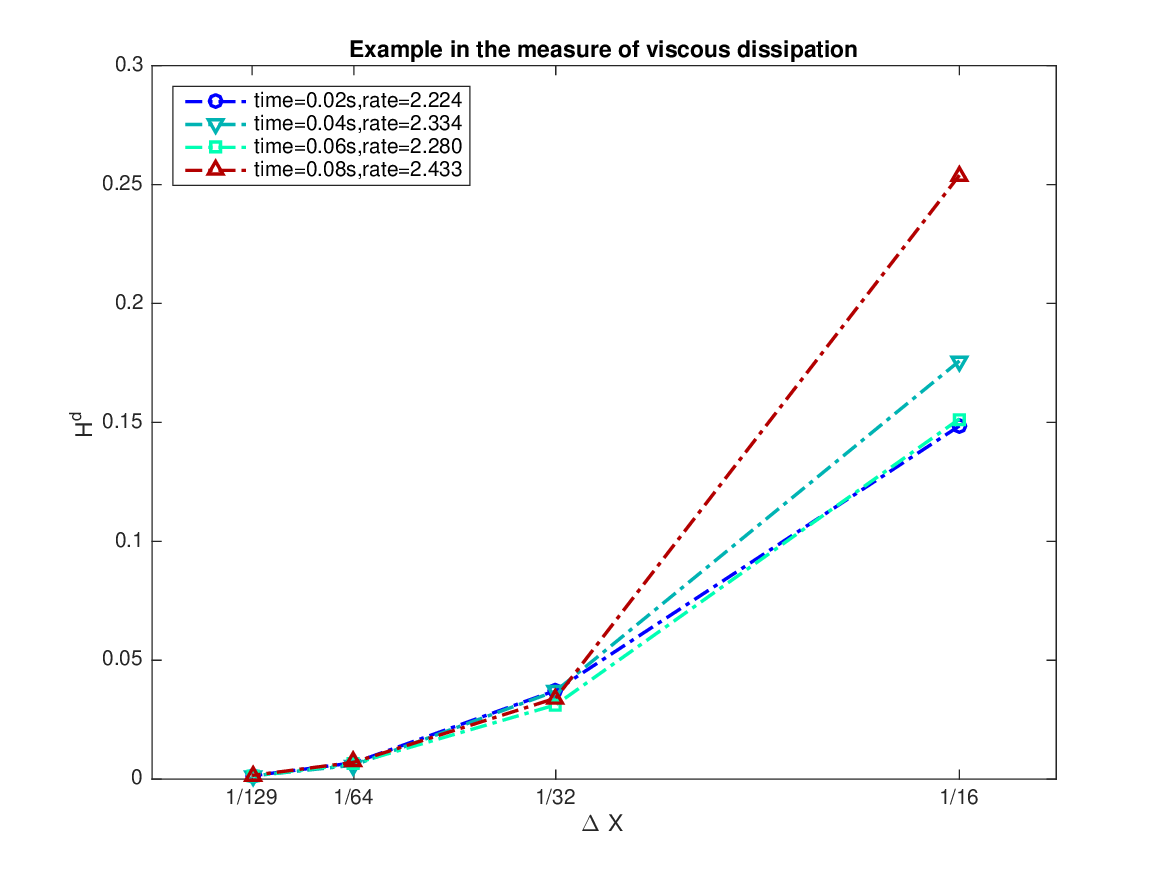}
	\caption{$H^d$  vs. $\Delta X$ in the viscous dissipation sense for $\Delta X = \frac{1}{M}, \ M \in \{16, 32, 64, 128\}$ at $t=i\Delta t, \ i = 1,\cdots, 4$, and $\Delta t = 0.02$ with $s_1=1$, $s_2=1.15>r_0=1.1087$, $\delta = 0.1$, and $\mu=0.01$.}
	\label{figs:convVsd1-15}
\end{center}
\end{figure}

\begin{figure}[htb]
\begin{center}
	\includegraphics[scale=0.6]{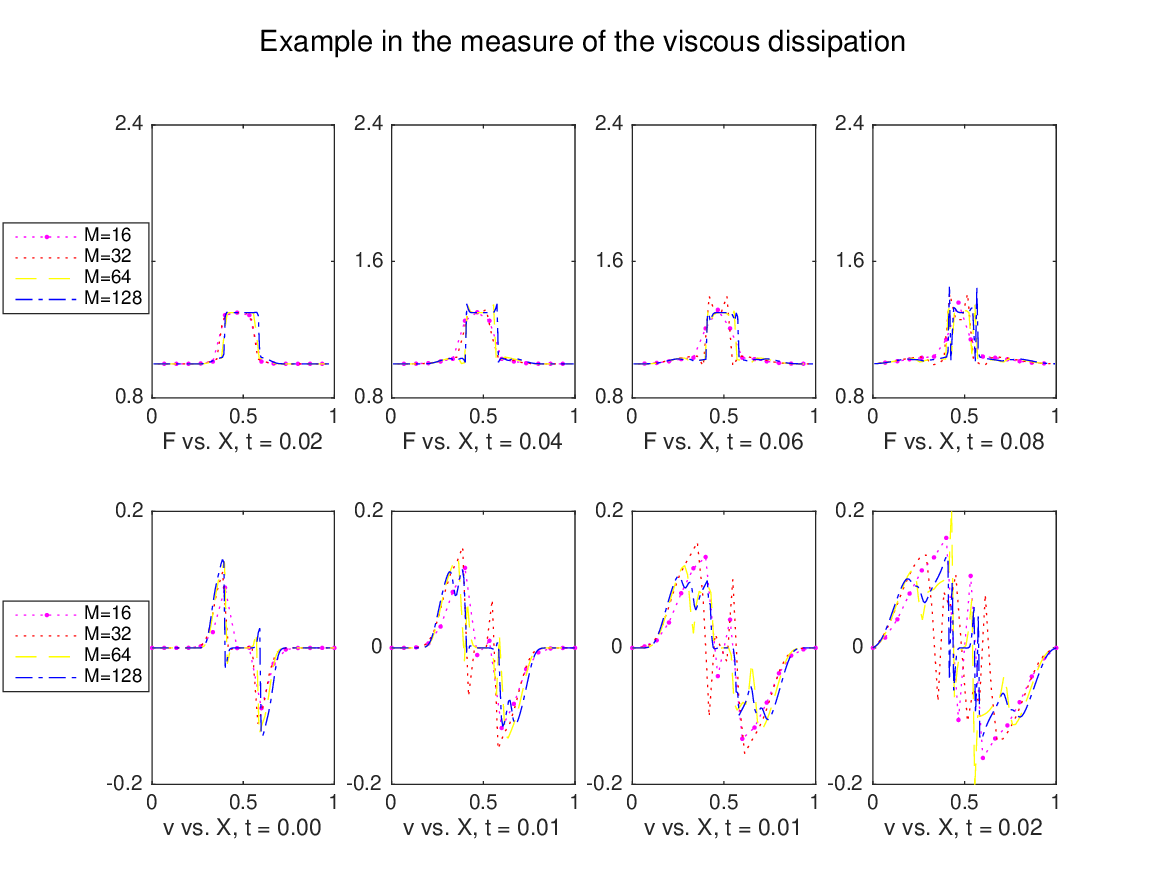}
	\caption{Numerical results with the viscous dissipation: $F$  vs. $X$ (top line) and $v$ vs. $X$ (bottom line) for $t=i \Delta t$, \ $i=1, \cdots, 4$, and $\Delta t=0.02$ with $s_1=1$, $s_2=1.3>r_0=1.1087$, $\delta = 0.1$, and $\mu=0.01$.}
	\label{figs:vsd1-3}
\end{center}
\end{figure}

The oscillations in Figure \ref{figs:vsd1-1} propagate near shocks as time increases, where the initial values of \eqref{eq:initialvalueform} are $s_1=1$ and $s_2=1.1$. The Riemann solution illustrated by Figure \ref{figs:rsolua} is computed by \eqref{eq:curveR2} and \eqref{eq:soluF} for $0\leq t<t_0$ as in \cite{Linshu2002}, and through \eqref{eq:curveS2} we calculate the Riemann solution for $t_0\leq t < T$. 
Since the coarse-grained model \eqref{eq:dis-viscous} is an approximation of the Riemann problem in the theory of hyperbolic conservation laws, the Riemann solution serves as the reference solution $\phi_R$ to compute $H^d$. Physically speaking, this comparison may indicate how well the coarse-grained solution could simulate the macroscopic feature of the material (the Riemann solution of the continuum model). 

We choose $\delta = 0.1,\ \rho_0=1,\ 3A=0.25,\ \mu=0.01$, and $M\in\{16,\;32,\;64,\;128\}$ in our numerical experiments. 
In Figure \ref{figs:vsd1-1}-\ref{figs:convVsd1-1}, we choose $s_1=1$ and $s_2=1.1$, the solution $U=(F,v)$ is always in the hyperbolic region. Figure \ref{figs:vsd1-1} describes that the coarse-grained model can approximate the main feature of the continuum model. 
We could see from Figure \ref{figs:convVsd1-1} that $H^d$ decreases with increasing $M$, which means the coarse-grained solutions get closer to the Riemann solution with greater $M$. 
In Figures \ref{figs:vsd1-15}-\ref{figs:convVsd1-15}, we keep $s_1 = 1$, and choose $s_2=1.15$ which is slightly greater than $r_0=1.1087$. At the beginning, $U$ lies partially ($X \in [1/2 - \delta, 1/2 + \delta]$) in the elliptic region. The solution goes quickly from the elliptic region to the hyperbolic region and then stays in the hyperbolic region. In this case, the coarse-grained model can still mimic the macroscopic feature. 
In Figures \ref{figs:vsd1-3}, we choose $s_1=1$ and $s_2=1.3$, the initial value of $U$ is partially in the elliptic region. 
We can see that an approximately stationary shock occurs. 
These results indicate that if the initial deformation gradients are smaller than $r_0$ or not significantly larger than $r_0$, the coarse-grained model in the measure of the viscous dissipation is consistent and reproduces the continuum solution pretty well. However, when the deformation is sufficiently large, this coarse-grained approximation is unreliable.

\subsubsection{Coarse-grained solution in the measure of a space average}
\label{sec:spaceAverage}

Theorems \ref{thm:dis-Ed} and \ref{thm:4.3} ensure the existence of a global solution of \eqref{eq:dis-pde}. Now we turn to another type of measure of the coarse-grained solution. Hou and Lax \cite{Hou1991} pointed out that the weak limits of $F$ and $v$ were determined by a suitable space average. Convoluting $F$ as a function of $X$ with a mollifier $g_{\tau}$, we obtain the average value of $F$ denoted as $F_{\textrm{avg}}$, 
\begin{equation}
  F_{\textrm{avg}} = F\ast g_{\tau},
  \label{eq:favg}
\end{equation}
where 
\begin{displaymath}
  g_{\tau}(X) s= \frac{1}{\tau}g\left(\frac{X}{\tau}\right),
\end{displaymath}
and
\begin{displaymath}
    g(X) = \left\{ \begin{array}{ll} 
    1+\cos(2\pi X),& -\frac{1}{2}\leq X \leq \frac{1}{2}; \\  
    0, & \textrm{otherwise}.
    \end{array} \right. 
\end{displaymath}

The value of $\tau$ is determined experimentally to make the support of $g_{\tau}$ small with respect to 1, and large with respect to $\Delta X$. In numerical experiments, we take $\tau = 4$ for $M=16\mbox{ or }32$, and $\tau = 6$ for $M=64 \mbox{ or }128$.

\begin{figure}[htb]
\begin{center}
	\includegraphics[scale=0.6]{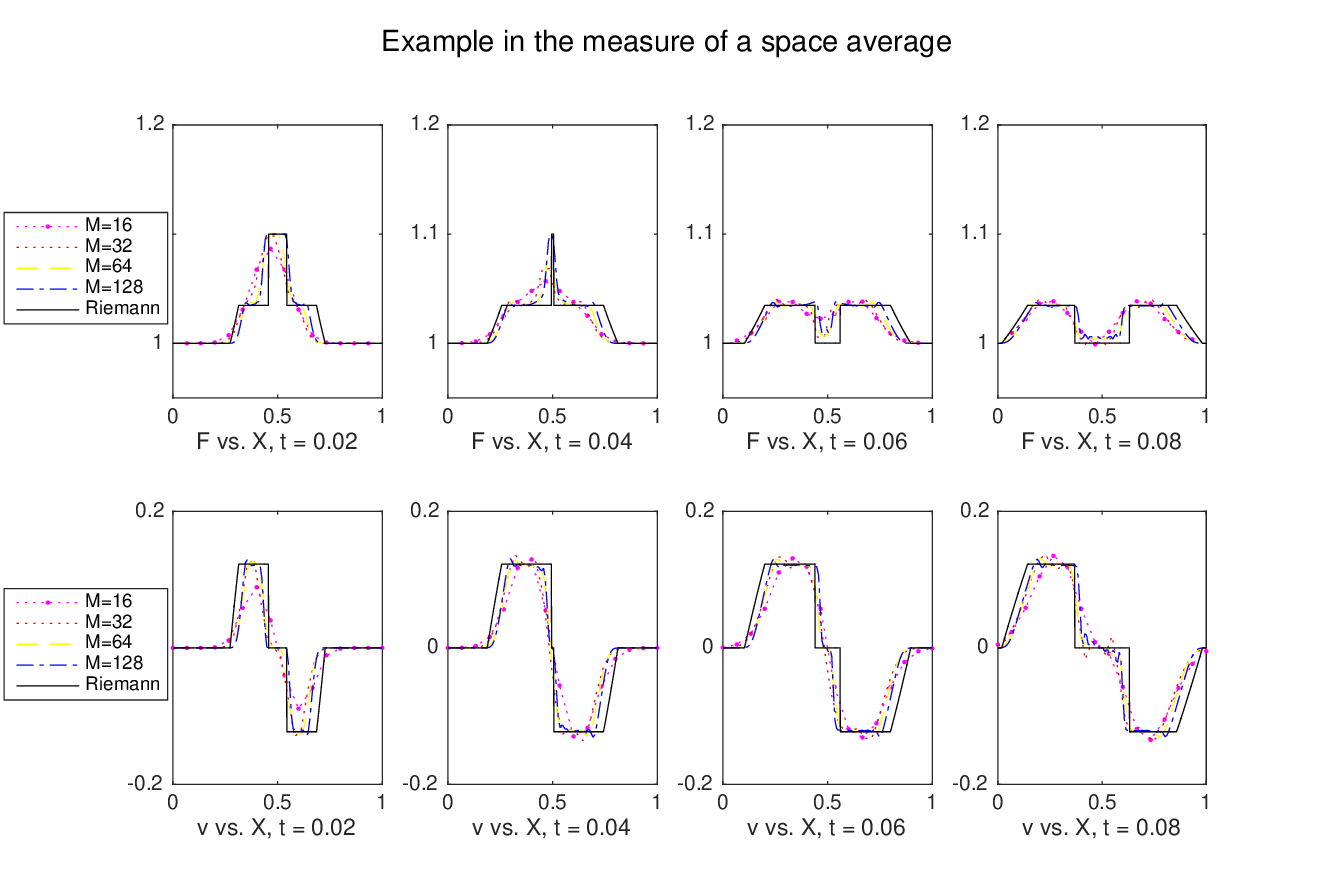}
	\caption{Numerical results with a space average: $F$  vs. $X$ (top line) and $v$ vs. $X$ (bottom line) for $t=i \Delta t$, \ $i=1,\cdots, 4$, and $\Delta t=0.02$ with $s_1=1$, $s_2=1.1<r_0=1.1087$, and $\delta = 0.1$.}
	\label{figs:avg1-1}
\end{center}
\end{figure}

\begin{figure}[htb]
\begin{center}
	\includegraphics[scale=0.6]{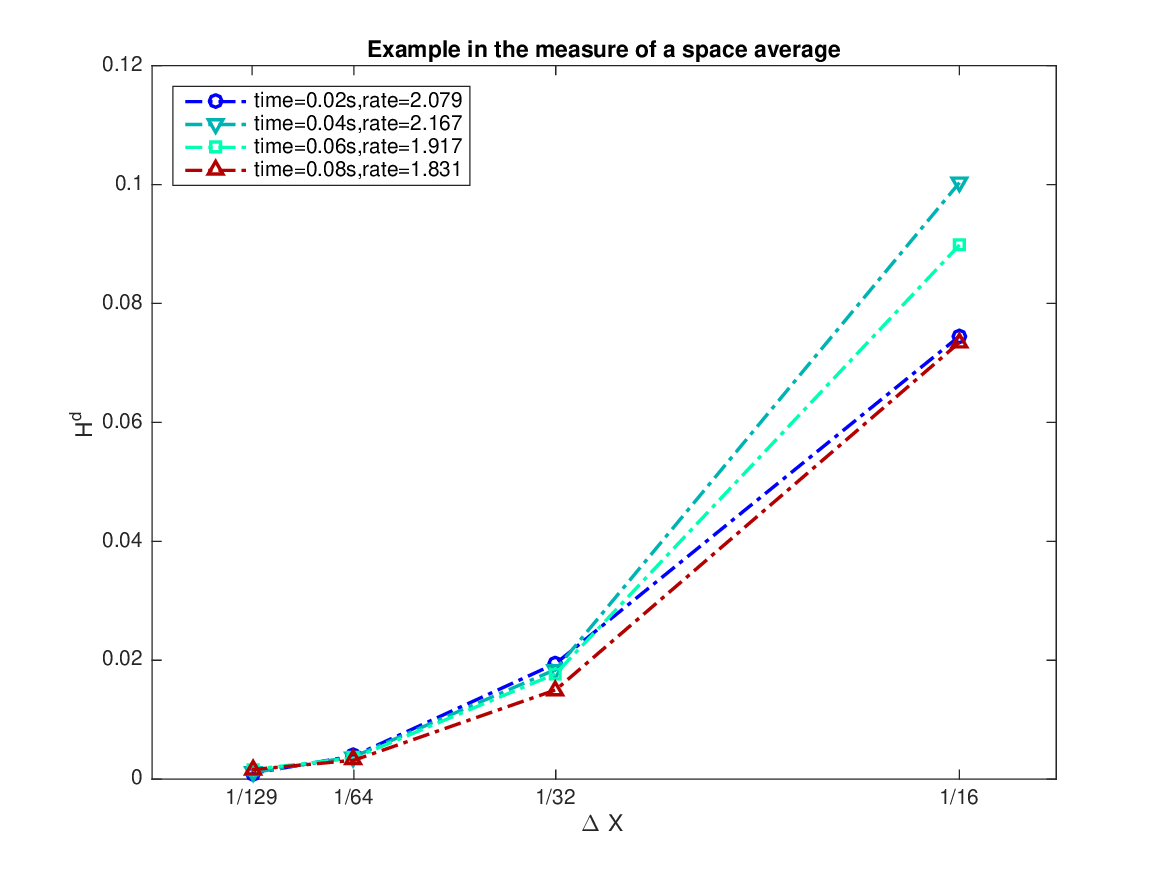}
	\caption{$H^d$  vs. $\Delta X$ in the space average sense for $\Delta X = \frac{1}{M}, \ M \in \{16,\; 32,\; 64, \;128\}$ at $t=i\Delta t, \ i = 1,\cdots, 4$, and $\Delta t=0.02$ with $s_1=1$, $s_2=1.1<r_0=1.1087$, and $\delta = 0.1$.}
	\label{figs:convAvg1-1}
\end{center}
\end{figure}

\begin{figure}[htb]
\begin{center}
	\includegraphics[scale=0.6]{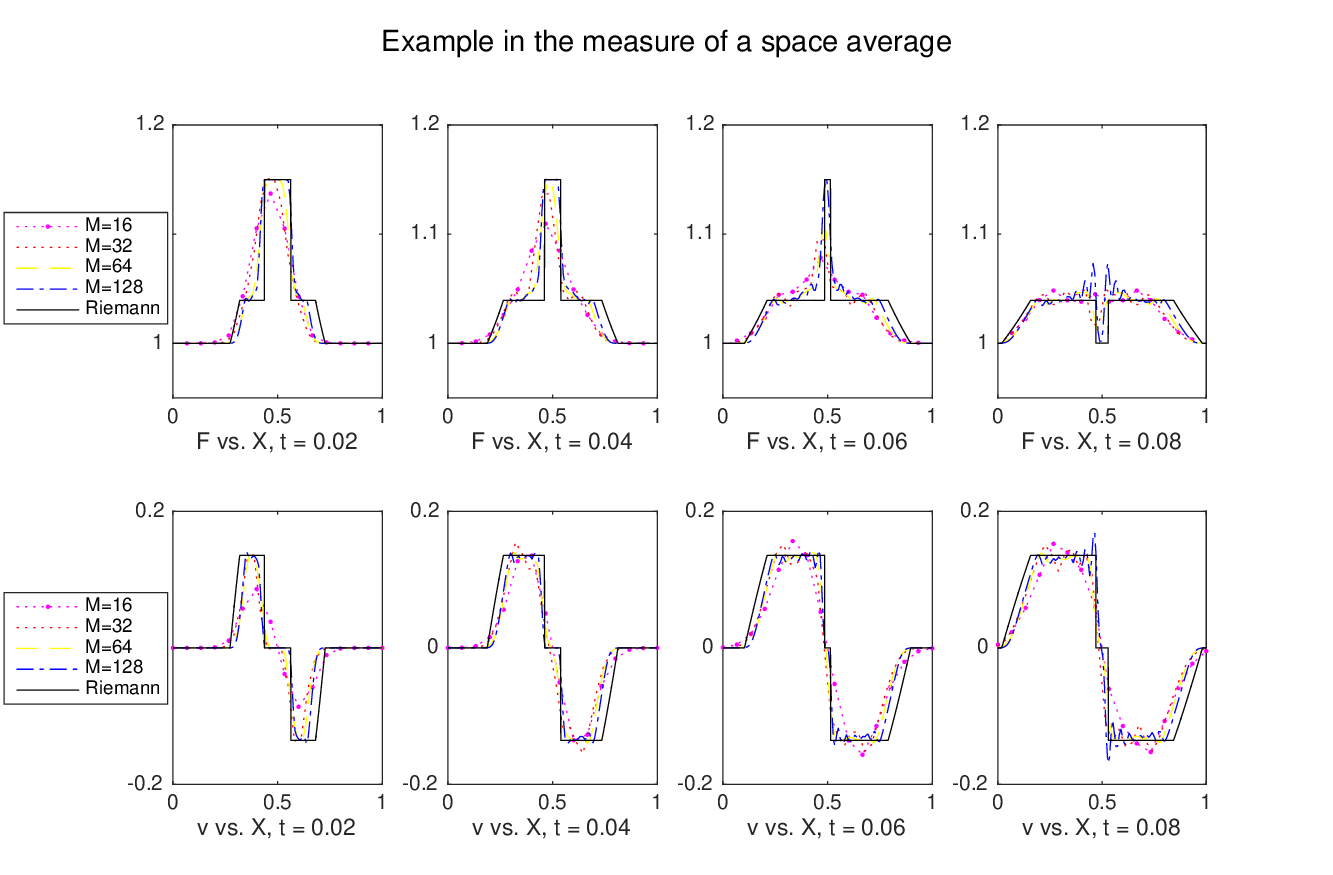}
	\caption{Numerical results with a space average: $F$  vs. $X$ (top line) and $v$ vs. $X$ (bottom line) for $t=i \Delta t$, \ $i=1, \cdots, 4$, and $\Delta t=0.02$ with $s_1=1$, $s_2=1.15>r_0=1.1087$, and $\delta = 0.1$.}
	\label{figs:avg1-15}
\end{center}
\end{figure}

\begin{figure}[htb]
\begin{center}
	\includegraphics[scale=0.6]{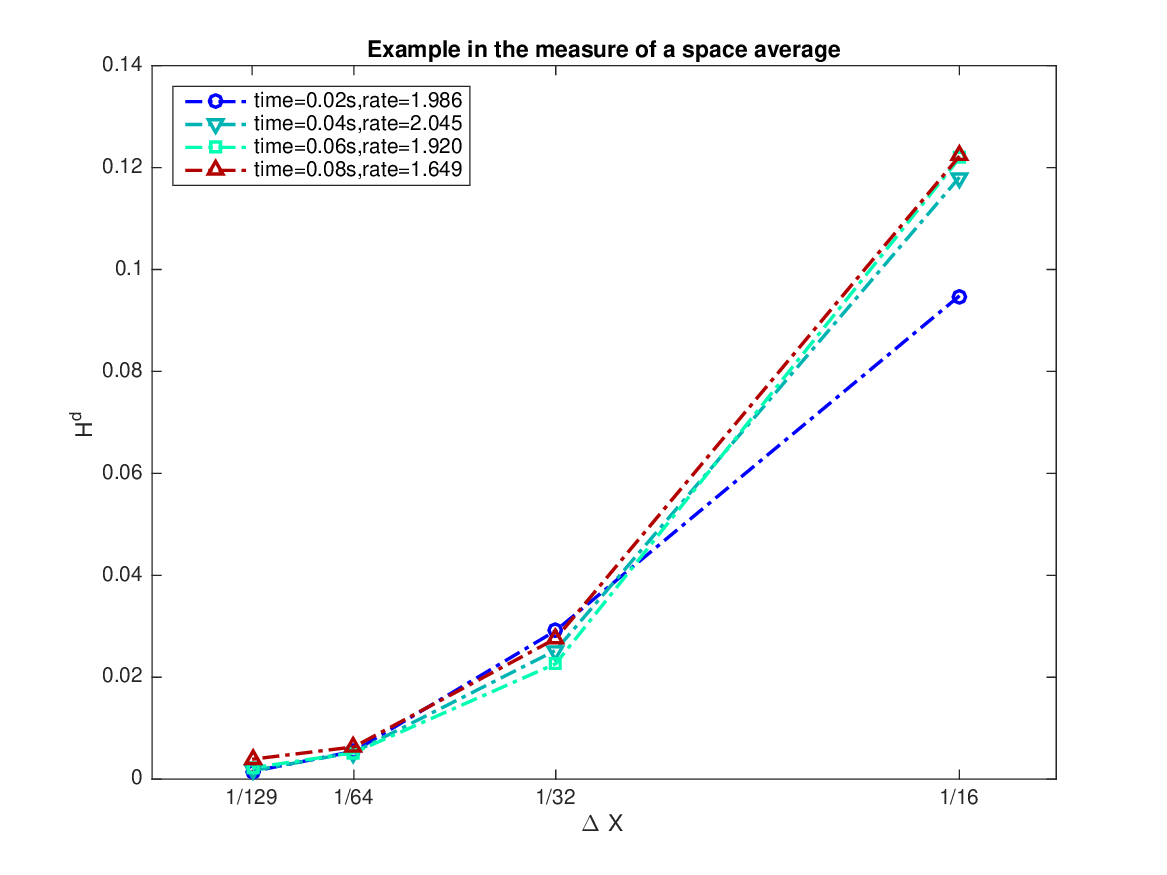}
	\caption{$H^d$  vs. $\Delta X$ in the space average sense for $\Delta X = \frac{1}{M}, \ M \in \{16,\; 32,\; 64, \;128\}$ at $t=i\Delta t, \ i = 1,\cdots, 4$, and $\Delta t = 0.02$ with $s_1=1$, $s_2=1.15>r_0=1.1087$, and $\delta = 0.1$.}
	\label{figs:convAvg1-15}
\end{center}
\end{figure}

\begin{figure}[htb]
\begin{center}
	\includegraphics[scale=0.6]{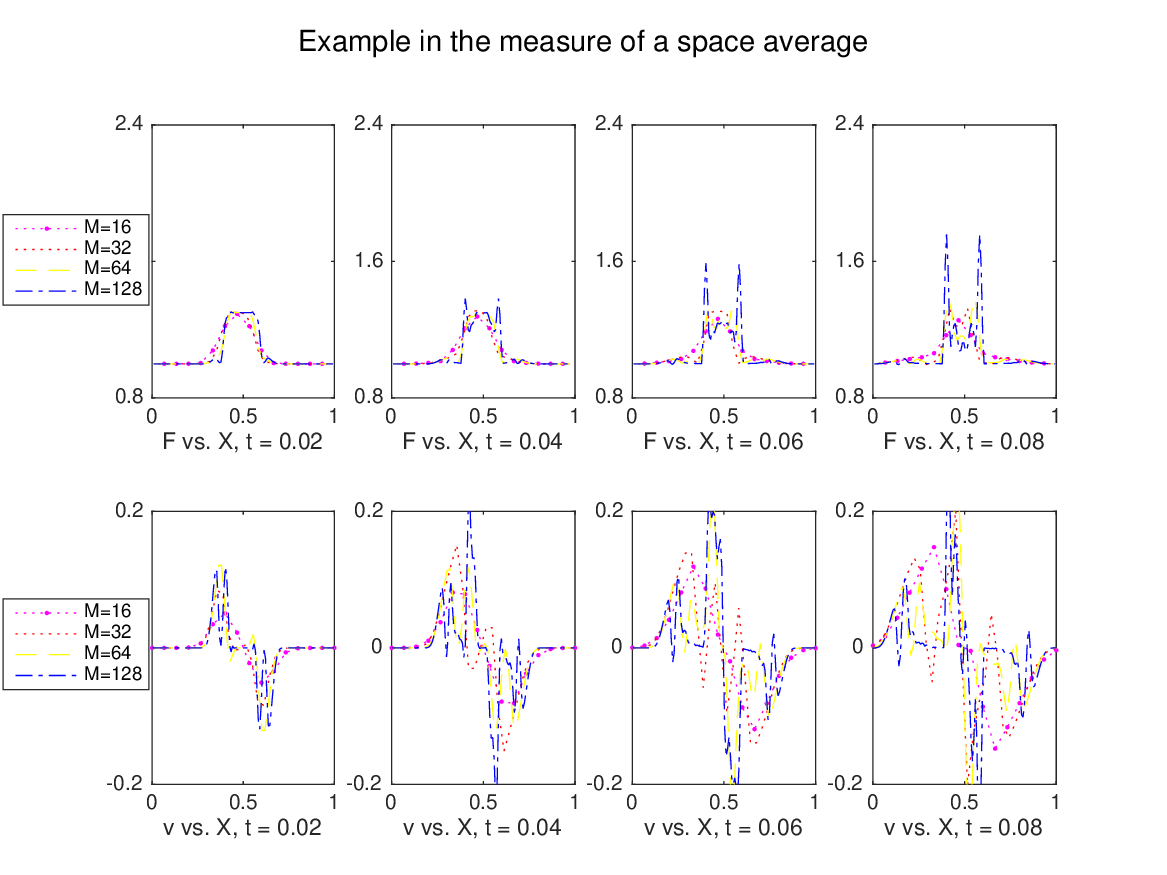}
	\caption{Numerical results with a space average: $F$  vs. $X$ (top line) and $v$ vs. $X$ (bottom line) for $t=i \Delta t$, \ $i=1, \cdots, 4$, and $\Delta t=0.02$ with $s_1=1$, $s_2=1.3>r_0=1.1087$, and $\delta = 0.1$.}
	\label{figs:avg1-3}
\end{center}
\end{figure}

When considering the solution in the measure of a space average, the MD solution ought to be the reference solution for the consistency study. However, in numerical experiments, as we did in the measure of the viscous dissipation in previous subsection, the coarse-grained solutions in the sense of a space average are also compared with the Riemann solution, although there is no theoretical result about their relationship. In addition, this comparison may also indicate how well the coarse-grained solution of the atomistic model could approximate the macroscopic feature of the material (the Riemann solution of the continuum model).

We choose the same parameters as in the examples measured with the viscous dissipation in \S~\ref{sec:viscousDamping}, i.e. $\delta = 0.1,\ \rho_0=1,\ 3A=0.25$, and $M\in\{16,\;32,\;64,\;128\}$.
In Figure \ref{figs:avg1-1}-\ref{figs:convAvg1-1}, choose $s_1=1$ and $s_2=1.1<r_0=1.1087$ so that $U=(F,v)$ stays in the hyperbolic region. 
In Figures \ref{figs:avg1-15}-\ref{figs:convAvg1-15}, we keep $s_1 = 1$, and choose $ s_2=1.15>r_0$ to make the initial values of $U$ lie in the elliptic region when $X \in [1/2 - \delta, 1/2 + \delta]$. 
For $U$ is either always in the hyperbolic region or in the elliptic region for a short time, the solutions in the measure of a space average can simulate the main properties of the continuum model, and the solutions get closer (distances are measured with $H^d$) to the Riemann solution with decreasing $\Delta X$.
In Figures \ref{figs:avg1-3}, choose $s_1=1$ and $s_2=1.3$. An approximately stationary shock occurs. 
These results indicate that the coarse-grained model in the measure of the space average is consistent and approximates the macroscopic feature pretty well if the initial deformation gradients are smaller than $r_0$ or not significantly larger than $r_0$. However, when the deformation is sufficiently large, this coarse-grained approximation is unreliable.

\section{Conclusion}
\label{sec:conclusion}

In this paper, we consider a coarse-grained approximation based on the virtual internal bond method. We focus especially on the consistency of the coarse-grained model with respect to the grain (or mesh) size. 
Since the solution of continuum model represents the macroscopic properties of the material, a defect under a uniform large deformation may be related to the Riemann problem of the continuum model (a nonlinear wave equation of mixed type).
By comparing the solutions between the Riemann problem of the continuum model and the coarse-grained model under certain measures we examine how well the coarse-grained model can approximate the macroscopic properties of the material.
 
According to our numerical results, the coarse-grained model using either the viscous dissipation or the space average seems to be a consistent approximation of the molecular dynamics model in the relatively small deformation with sufficient large $M$. But with a large deformation, there is no good approximation to the MD model in either the viscous dissipation sense or the space average sense. 

The instability of the linearized model does not prevent the existence and boundedness of the solution of the original nonlinear model even if $s_2$ is slightly greater than $r_0$. However, it may be the reason of the inconsistency of the coarse-grained model we observed in this study when $s_2$ is sufficiently greater than $r_0$.

This opens an avenue for further studies of longer range interactions, atomistic to continuum coupling model, and higher dimensional problems. Especially the problem that couples the nonlinear wave equation in a small deformation region (hyperbolic region) to an atomistic level model for a large deformation (elliptic region) to simulate the fracture dynamics. We refer to \cite{Slemrod1983,Shearer1993,Slemrod1989} and \cite{shearer93} for other choices of the dissipative term for zero temperature simulation in the future. Moreover, since the hot-QC method (cf. \cite{Blanc2010,Tadmor2013}) and the hyper-QC method (cf. \cite{Kim2014}) constructed a correction to obtain a second order accuracy in temperature, the correction offers us an idea to construct a temperature related dissipative term. The construction of the solution of Riemann problem for higher dimensions is highly challenging, but numerical studies could certainly be extended to higher dimensions and finite range interactions.

\bibliographystyle{plain}
\bibliography{nonWave.bib}

\end{document}